\documentclass[11pt]{article}
\usepackage[utf8]{inputenc}
\usepackage{multirow}
\usepackage{tikz}
\usetikzlibrary{positioning}
\usepackage[a4paper, margin=2.3cm]{geometry}
\usepackage{amsmath,amssymb,amsthm,mathtools,mathrsfs}
\usepackage{thmtools}
\usepackage{thm-restate}

\usepackage{xcolor}
\usepackage{parskip}
\usepackage{float}
\usepackage[hidelinks,hyperfootnotes=false]{hyperref}
\usepackage{tikz-cd}
\usetikzlibrary{matrix,arrows.meta}

\usepackage{setspace}
\usepackage{makecell}

\usepackage{graphicx}
\usepackage{enumitem}
\usepackage{comment}

\newcommand{\F}{\mathbb{F}}
\newcommand{\C}{\mathbb{C}}
\newcommand{\1}{\mathbf{1}}

\usepackage{booktabs,tabularx,needspace}

\newtheorem{theorem}{Theorem}[section]

\newtheorem{corollary}[theorem]{Corollary}

\newtheorem{lemma}[theorem]{Lemma}

\newtheorem{proposition}[theorem]{Proposition}

\newtheorem{conjecture}[theorem]{Conjecture}

\newtheorem*{definition*}{Definition}
\newtheorem*{remark*}{Remark}

\makeatletter
\def\theHtheorem{\thesection.\arabic{theorem}}
\@for\ARevEnv:=remark,corollary,problem,lemma,question,proposition,definition,conjecture,example,claim,assumption\do{
  \expandafter\def\csname theH\ARevEnv\endcsname{\theHtheorem}}
\makeatother

\usepackage[T1]{fontenc}
\usepackage{lmodern}
\allowdisplaybreaks[2]
\usepackage{etoolbox}
\BeforeBeginEnvironment{proof}{\Needspace{5\baselineskip}}
\BeforeBeginEnvironment{conjecture}{\Needspace{10\baselineskip}}
\hypersetup{pdftitle={Sharp spherical extension theorem in F_q^(2m) and applications},pdfauthor={Thang Pham and Boqing Xue}}
\title{Sharp spherical extension theorem in $\F_q^{2m}$ and applications}
\author{Thang Pham \thanks{Institute of Mathematics and Interdisciplinary Sciences, Xidian University. \newline
\hspace*{0.45cm} Email: {\tt thangphammath@xidian.edu.cn}} \and Boqing Xue \thanks{Institute of Mathematical Sciences, ShanghaiTech University. ~Email: {\tt xuebq@shanghaitech.edu.cn}}}
\date{}

\begin{document}

\maketitle

\begin{abstract}
Let $q$ be an odd prime power and $m\geq 1$. Let $Q$ be a nondegenerate quadratic form on $\mathbb F_q^{2m}$. For every sphere $S_j=\{Q=j\}$ with $j\in \F_q^\times$, we prove the sharp extension estimate $R_{S_j}^*(2\to r)\lesssim_m1$ for $r\ge 2(m+1)/m$, uniformly in $q$, $Q$, and $j$. As an application, we show that if $E\subseteq\mathbb F_q^{2m}$ satisfies $|E|/q^{m+1/3}\to\infty$, then almost every pin $y\in E$ determines $(1-o(1))q$ values of $Q(x-y)$. The proof combines two arithmetic Hecke operator estimates with an induction in the dimension, a centered sphere--cone estimate, and an orthogonal decomposition.
\end{abstract}

\tableofcontents

\section{Introduction}\label{sec:main}
\subsection{Background and main results}

Let $q$ be an odd prime power, write $\F=\F_q$, and fix an integer
$m\ge1$. Throughout the paper, $d=2m$ and $Q$ is a nondegenerate
quadratic form on $V=\F^d$. Its associated symmetric bilinear form is
\[
 B(x,y)=\frac{Q(x+y)-Q(x)-Q(y)}2.
\]

For $j\in\F^\times$, put
\[
 S_j=S_j(Q)=\{x\in V:Q(x)=j\},\qquad
 d\sigma_j=|S_j|^{-1}\1_{S_j}.
\]
We call these nonzero level sets spheres. All ambient norms and
convolutions use counting measure, while surface norms use $d\sigma_j$. Thus,
\[
 \|g\|_{L^r(V)}=\left(\sum_{x\in V}|g(x)|^r\right)^{1/r},
 \qquad
 \|f\|_{L^u(S_j,d\sigma_j)}
 =\left(\frac1{|S_j|}\sum_{\xi\in S_j}|f(\xi)|^u\right)^{1/u},
\]
with the usual modifications at infinity.

Let $\chi$ be the canonical nontrivial additive character of $\F$. We use the bilinear form $B$ to identify $V$ with its dual, and define
\[
 \widehat g(\xi)=\sum_{x\in V}g(x)\chi(-B(x,\xi)),
 \qquad
 (f\,d\sigma_j)^\vee(x)
 =\frac1{|S_j|}\sum_{\xi\in S_j}f(\xi)\chi(B(x,\xi)).
\]
This convention gives the same global extension constants as the standard dot-product convention, by an invertible change of physical coordinates.

We write
$R_{S_j}^*(u\to r)\lesssim_{m,u,r}1$ if there is a constant $C_{m,u,r}>0$, uniformly in $q$, $Q$, and $j\ne0$, such that
\begin{equation}\label{extension}
 \|(f\,d\sigma_j)^\vee\|_{L^r(V)}
 \le C_{m,u,r}\|f\|_{L^u(S_j,d\sigma_j)}
\end{equation}
for every complex function $f$ on $S_j$. By duality, this is equivalent to
\begin{equation}\label{restriction}
 \|\widehat g\|_{L^{u'}(S_j,d\sigma_j)}
 \le C_{m,u,r}\|g\|_{L^{r'}(V)}
\end{equation}
for every complex function $g$ on $V$, where $u'$ and $r'$ are the conjugate exponents of $u$ and $r$, respectively.

Mockenhaupt and Tao~\cite{MT04} introduced the finite-field Fourier
restriction problem and its connections with Kakeya and incidence
geometry. Subsequent work has treated spheres, paraboloids, cones,
and related varieties. See, for example,
\cite{IK08,AK,Kohkang,KK2025,KKY,KohPhamVinh2021,Le13,european,Le15,RS}.
For the standard sphere defined by
$Q_0(x)=x_1^2+\cdots+x_d^2$, Iosevich and Koh~\cite{IK08} proved
the $L^2$ Stein--Tomas range $r\ge2(d+1)/(d-1)$. In even
dimensions the conjectured range is
\begin{equation}\label{ConjEven}
 R_{S_j}^*(2\to r)\lesssim_{d,r}1
 \quad\Longleftrightarrow\quad r\ge\frac{2(d+2)}d.
\end{equation}
The two-dimensional case was established in~\cite{IK08}. See also
\cite[Conjecture~1.4]{KK2025} for the even-dimensional formulation.
Our main result proves this range for every nondegenerate quadratic
form in even dimension over every odd finite field.

\begin{theorem} \label{thm:main-endpoint-three}

For every $m\ge1$,
\[
R_{S_j}^*(2\to r)\lesssim_m 1,\qquad r\geq \frac{2(m+1)}{m},
\]
where the implied constant is independent of $q$, $Q$ and $j\in \F^\times$.
\end{theorem}

The implied constant is also independent of $r$, since $L^r$ norms decrease with $r$ for counting measure.

This range is sharp for both Witt types.
The endpoint is $4,3,8/3,5/2$ in dimensions $2,4,6,8$, respectively.
Sharpness follows from an affine totally isotropic subspace of
dimension $m-1$ contained in $S_j$. Indeed, if $x_0\in S_j$, the
odd-dimensional nondegenerate space $x_0^\perp$ has Witt index $m-1$.
Choosing a maximal totally isotropic subspace $W\subseteq x_0^\perp$
gives $x_0+W\subseteq S_j$. For its indicator, the ratio of the extension norm to the surface norm is comparable to $q^{(m+1)/r-m/2}$, which forces the stated range.
For $m=1$, this is the single-point example.

For $E\subseteq\F_q^{2m}$ and $y\in\F_q^{2m}$, write
\[
 \Delta_y(E)=\{Q(x-y):x\in E\}.
\]
Here $m$ is fixed and $q$ tends to infinity through all odd prime powers. Under a size condition with a fixed $\varepsilon>0$, the error terms below are uniform over $Q$ and the sets satisfying the stated size conditions. For a hypothesis expressed by a ratio tending to infinity, the errors also depend on that ratio.

\begin{corollary} \label{cor:same-set-seven-thirds}
Let $E\subseteq\F_q^{2m}$ satisfy $|E|/q^{m+1/3}\longrightarrow\infty$. Then there exists $E'\subseteq E$ such that $|E'|=(1-o(1))|E|$ and $|\Delta_y(E)\setminus\{0\}|=(1-o(1))q$ uniformly for every $y\in E'$.
\end{corollary}

In particular, the conclusion holds under
$|E|\geq q^{m+1/3+\varepsilon}$ for any fixed $\varepsilon>0$,
with errors $o_\varepsilon(1)$ uniform in the characteristic and
extension degree. The threshold is $4/3$ in dimension two and
$7/3$ in dimension four.

The global endpoint gives an asymmetric statement without any
additional restriction on the number of pins.

\begin{corollary}\label{cor_distance_2}
Let $\varepsilon>0$ and let $E,F\subseteq\F_q^{2m}$. If $|E||F|^{m/(m+1)}\geq q^{2m+\varepsilon}$, then there exists $F'\subseteq F$ such that
$|F'|=(1-o_\varepsilon(1))|F|$ and
$|\Delta_y(E)\setminus\{0\}|=(1-o_\varepsilon(1))q$
uniformly for every $y\in F'$.
\end{corollary}

The same conclusion holds if
$|E||F|^{m/(m+1)}/q^{2m}\to\infty$.
For $F=E$, this gives the exponent $2m(m+1)/(2m+1)$.
The orthogonal decomposition yields the following stronger result in the middle range when $m\geq2$.

\Needspace{9\baselineskip}
\begin{corollary}\label{cor:improved-middle-pins}
Let $\varepsilon>0$ and let $E,F\subseteq\F_q^{2m}$ satisfy
\[
 q^m\leq |F|\leq q^{m+1},
 \qquad |E||F|^{1/2}\geq q^{(3m+1)/2+\varepsilon}.
\]
Then there exists $F'\subseteq F$ such that
$|F'|=(1-o_\varepsilon(1))|F|$ and
$|\Delta_y(E)\setminus\{0\}|=(1-o_\varepsilon(1))q$
uniformly for every $y\in F'$.
\end{corollary}

\subsection{Main ideas}\label{subsec:proof-overview}

We describe the proof with counting measure on $S_j$ and on $V$.
Put $T f(x)=\sum_{\xi\in S_j}f(\xi)\chi(B(x,\xi))$ and
$\mathcal F=q^{-m}T$, with functions extended by zero off $S_j$.
Thus, $\mathcal F$ is the unitary ambient Fourier transform.

\paragraph{Shifted spheres and arithmetic operators.}
The convolution identity in Section~\ref{sec:exact-product-reduction}
reduces the problem to the centered matrices
\[
 M_t(x,y)=\1_{S_j}(x+y-t)-\frac{|S_j|}{q^{2m}},
 \qquad x,y\in S_j.
\]
We prove $\|M_t\|\lesssim_m q^{m-1}$ whenever $Q(t)\ne j$.
For $m\ge2$, choose a hyperbolic plane and write
$Q(b,c,z)=bc+P(z)$. On the part $c=0$, the matrix reduces to a
lower-dimensional shifted-sphere matrix, with the cross terms bounded
by an elementary Gram-matrix calculation. On $c\ne0$, Fourier
transformation in the $z$ variables reduces the remaining estimates
to matrices involving at most two scalar indices.

For a nonisotropic translation, Gaussian summation expresses the
oscillatory matrix as a weighted sum of finite Weil operators.
The operators form a representation of $\mathrm{SL}_2(\mathcal K)$
for a two-dimensional commutative algebra $\mathcal K$ over $\F_q$.
An exact realization as a compression of a degree-one automorphic
Hecke operator over $\F_q(X)$ gives the bound $2\sqrt q$ for the
weighted sum. The coefficient spaces, stabilizer weights, and
counting norms are included in this identification.

For a nonzero isotropic translation, the remaining matrix is a
Kloosterman-type matrix indexed by one field variable. First-jet Hecke
correspondences give its norm bound outside characteristic $3$.
In characteristic $3$, a separate Fourier transformation and
mixed character sums give a bound of the same order.
For $t=0$, we use Fourier transformation in both variables. In characteristic three, the matrix is a convolution operator on the additive group of the field. These arguments make the induction uniform in the
characteristic and the extension degree.

\paragraph{Why the sphere--cone estimate must be centered.}
For translations $t\in S_j$, projective polarity controls the part
where $Q(x-t)\ne0$. The remaining part involves the cone $S_0$.
The appropriate estimate in every even dimension is
\[
 \|\1_{S_0}(f*h)\|_2^2
 \le q^{-1}\|f*h\|_2^2+q^{m-1}\|f\|_2^2\|h\|_2^2,
 \qquad \operatorname{supp}f,\operatorname{supp}h\subseteq S_j.
\]
The density term $q^{-1}\|f*h\|_2^2$ is essential in dimensions
at least six. We prove the estimate by embedding the relevant
quadratic level sets into projective quadrics and computing the
spectrum of their orthogonality matrices. Degenerate sections are
handled by passing to the quotient by their radical.

\paragraph{Fourth moments and orthogonal decomposition.}
Combining these estimates gives the refined fourth moment
\[
 \|\mathcal F f\|_4^4
 \lesssim_m q^{-1}\|\mathcal F f\|_\infty^2\|f\|_2^2
                   +q^{-m-1}\|f\|_2^4.
\]
Fix a sufficiently large absolute constant $A>0$. For a function $f$ with $\|f\|_2=1$, successively project onto
the normalized characters
\[
 a_x(\xi)=|S_j|^{-1/2}\chi(-B(x,\xi))
\]
whenever the remaining
Fourier coefficient exceeds $Aq^{m/2}$. At most $O_m(q^{m-1})$
characters are selected. Their Gram matrix is uniformly invertible,
because the off-diagonal inner products are $O(q^{-m+1/2})$.
The resulting remainder $r$ has
$\|\mathcal F r\|_4^4\lesssim_m q^{-m-1}$. Interpolation with
Plancherel gives its $L^{2(m+1)/m}$ bound. The selected characters
are estimated using their individual peaks and the bounded
off-diagonal Fourier kernel. This proves the strong endpoint without
a logarithmic loss.

For the distance applications, we insert the decomposition into an
exact pinned-energy identity and sum over the nonzero Fourier
spheres. The individual peaks are controlled in $L^2$, the
other contributions from the selected characters in $L^\infty$, and the remainder in $L^4$.
Together with the sphere kernel and Plancherel, this gives the
asymmetric distance bounds and the exponent $m+1/3$. The
contribution of the zero sphere is estimated separately for the two
Witt types.

Section~\ref{sec:setup} records the Fourier and arithmetic conventions.
Section~\ref{sec:exact-product-reduction} gives the convolution identity.
Section~\ref{sec:hecke-matrix-estimates} proves the arithmetic bounds
and the induction for shifted spheres.
Section~\ref{sec:fourth-moments-sharp-extension} proves the centered
sphere--cone estimate, the fourth moment, and the endpoint.
Section~\ref{subsec:optimized-pinned-distances} proves the unconditional
pinned distance theorems directly from these estimates. Section~\ref{section_app} states a stronger localized extension conjecture over prime fields and proves its conditional distance consequences.

\subsection*{Scope of the results}

All spherical extension and restriction statements concern
nonzero radii. The zero sphere is used only as an auxiliary set in
the convolution and distance arguments. Characteristic two is not
included: polarization and the Gaussian and Weil formulas used here
require that $2$ be invertible.

Restriction estimates also have applications to incidence and projection problems \cite{KP.22,KLP22,HHHKP}.

\section{Fourier and arithmetic preliminaries}\label{sec:setup}

We first fix the Fourier conventions and record the Gaussian formulas in
even dimension. We then state the arithmetic operator bounds used in
the matrix reduction.

\subsection{Fourier conventions and quadratic forms}

Throughout, $q$ is an odd prime power, $\F=\F_q$, and
$p=\operatorname{char}\F$. We use the canonical nontrivial additive character
\[
 \chi(t)=\exp\!\left(\frac{2\pi i}{p}
                 \operatorname{Tr}_{\mathbb F_q/\mathbb F_p}(t)\right).
\]
Unless a measure is specified, norms and convolutions use
counting measure. Our inner product is linear in the first variable:
\[
 \langle f,g\rangle=\sum_x f(x)\overline{g(x)},
 \qquad (f*g)(z)=\sum_x f(x)g(z-x).
\]
Fix an integer $m\ge1$ and a nondegenerate quadratic form $Q$ on
$V=\F^{2m}$. Its associated symmetric bilinear form is
\[
 B(x,y)=\frac{Q(x+y)-Q(x)-Q(y)}2.
\]
Thus, $B(x,x)=Q(x)$. Orthogonality and $W^\perp$ refer to $B$, i.e., $W^\perp=\{x:B(x,w)=0\text{ for every }w\in W\}$.
In coordinates, write $Q(x)=x^{\mathsf T}Ax$, where $A$ is invertible
and symmetric, and put
\[
 \epsilon_Q=\eta((-1)^m\det A)\in\{1,-1\}.
\]
Here $\eta$ is the quadratic character of $\F^\times$. Changing coordinates multiplies $\det A$ by a square, so $\epsilon_Q$
does not depend on the chosen basis. A Witt decomposition identifies
$Q$ with $H^m$ when $\epsilon_Q=1$, and with $H^{m-1}$ plus an
anisotropic plane when $\epsilon_Q=-1$, where $H(u,v)=uv$.
We call the first case \emph{split}. For $m\ge2$, an isotropic form need not be split.

For $f:V\to\C$, the unnormalized Fourier transform and its inverse are
\[
 \widehat f(\xi)=\sum_x f(x)\chi(-B(x,\xi)),\qquad
 f(x)=q^{-2m}\sum_\xi\widehat f(\xi)\chi(B(x,\xi)).
\]
Thus, $\|\widehat f\|_2=q^m\|f\|_2$. For $j\ne0$, put $S_j=\{\xi:Q(\xi)=j\}$ and $d\sigma_j=|S_j|^{-1}\1_{S_j}$. One has $|S_j|=q^{2m-1}-\epsilon_Qq^{m-1}$. The cardinality formula is proved below. Extension is defined by
\[
 (f\,d\sigma_j)^\vee(x)
 =\frac1{|S_j|}\sum_{\xi\in S_j}f(\xi)\chi(B(x,\xi)).
\]
We also use
\begin{equation}\label{eq:unitary-transform-convention}
 Tf(x)=\sum_{\xi\in V}f(\xi)\chi(B(x,\xi)),
 \qquad \mathcal Ff=q^{-m}Tf.
\end{equation}
The nondegeneracy of $B$ makes $\mathcal F$ unitary. Moreover,
\begin{equation}\label{eq:unitary-convolution-convention}
 \mathcal F(f*g)=q^m(\mathcal Ff)(\mathcal Fg),\qquad
 (f\,d\sigma_j)^\vee=\frac{q^m}{|S_j|}\mathcal Ff,
 \qquad \|f\|_{L^2(S_j,d\sigma_j)}=|S_j|^{-1/2}\|f\|_2,
\end{equation}
where functions on $S_j$ are extended by zero. Also
$Tf(x)=\widehat f(-x)$.

\paragraph{Coordinate invariance.}
Suppose $Q(x)=Q'(Ux)$ for an invertible linear map $U$, and let $B'$
be half the polarization of $Q'$. Then
$B(x,\xi)=B'(Ux,U\xi)$ and $U(S_j(Q))=S_j(Q')$.
For $f_0(z)=f(U^{-1}z)$,
\[
(f\,d\sigma_j(Q))^\vee(x)=(f_0\,d\sigma_j(Q'))^\vee(Ux).
\]
The normalized surface norms agree, and sums over $E$ correspond to sums over $UE$.
With the standard dot-product Fourier pairing, the change of variables $x\mapsto Ux$ in this identity is replaced by $x\mapsto U^{-\mathsf T}x$.
The choice between $B$ and the dot product itself only reindexes physical space: since $B(x,\xi)=(Ax)\cdot\xi$, the extension defined using $B$ is obtained from the dot-product extension by the change of variables $x\mapsto Ax$.
Finally,
\[
 \{Q(x-y):x\in E\}=\{Q'(z-Uy):z\in UE\}.
\]
These identities preserve the global extension constants, cardinality
hypotheses, and counts of pins and distances.
They compare isometric forms; they do not identify the split and
nonsplit types. Both types are covered by the arguments below.

The standard form $Q_0(x)=x_1^2+\cdots+x_{2m}^2$ has
$\epsilon_{Q_0}=\eta((-1)^m)$. In particular, the sum of four squares
is split over every odd finite field. In general even dimension,
the results apply to the standard form regardless of its type.

\subsection{Gaussian sums and the sphere kernel}

Extend $\eta$ by $\eta(0)=0$, and put
$G_\eta=\sum_{u\in\F}\chi(u^2)$.

\begin{lemma}\label{lem:one-dimensional-gauss-sum}
One has $|G_\eta|=q^{1/2}$ and $G_\eta^2=\eta(-1)q$. For $s\ne0$,
\[
 \sum_{u\in\F}\chi(su^2+vu)
 =G_\eta\eta(s)\chi\left(-\frac{v^2}{4s}\right).
\]
\end{lemma}

\begin{lemma}\label{lem:quadratic-gauss-transform}
Let $C$ be an invertible symmetric $k$-by-$k$ matrix over $\F$.
For $s\ne0$ and $v\in\F^k$,
\[
 \sum_{u\in\F^k}\chi(su^{\mathsf T}Cu+v^{\mathsf T}u)
 =G_\eta^k\eta(s)^k\eta(\det C)
       \chi\left(-\frac{v^{\mathsf T}C^{-1}v}{4s}\right).
\]
In particular, for the form $Q$ on $V=\F^{2m}$,
\begin{equation}\label{eq:even-gauss-transform}
 \sum_{\xi\in V}\chi(sQ(\xi)+B(x,\xi))
 =\epsilon_Qq^m\chi\left(-\frac{Q(x)}{4s}\right).
\end{equation}
\end{lemma}

\begin{proof}
Complete the square, diagonalize $C$ by congruence, and apply
Lemma~\ref{lem:one-dimensional-gauss-sum} in each coordinate.
For the second formula, take $C=A$ and $v=Ax$.
Then $v^{\mathsf T}A^{-1}v=Q(x)$, $\eta(s)^{2m}=1$, and
$G_\eta^{2m}\eta(\det A)=\epsilon_Qq^m$.
\end{proof}

We use the classical twisted Kloosterman bound in the following form;
see \cite{Weil1948}.

\begin{lemma}\label{lem:twisted-kloosterman}
Let $\psi$ be a multiplicative character of $\F^\times$. If
$(a,b)\ne(0,0)$, then
\[
 \left|\sum_{s\ne0}\psi(s)\chi(as+b/s)\right|\le2\sqrt q.
\]
When $a=b=0$, the sum is $q-1$ if $\psi=1$, and zero otherwise.
\end{lemma}

\Needspace{10\baselineskip}
\begin{lemma}\label{lem:size-kernel}
For every nondegenerate $Q$ on $V=\F^{2m}$ and every $j\ne0$,
\begin{equation}\label{eq:unnormalized-sphere-kernel}
 \sum_{\xi\in S_j}\chi(B(x,\xi))
 =q^{2m-1}\1_{\{0\}}(x)
 +\epsilon_Qq^{m-1}\sum_{s\ne0}
      \chi\left(-js-\frac{Q(x)}{4s}\right).
\end{equation}
Consequently, $|S_j|=q^{2m-1}-\epsilon_Qq^{m-1}$ and
\begin{equation}\label{eq:normalized-sphere-kernel}
 (d\sigma_j)^\vee(x)
 =\frac{q^m}{q^m-\epsilon_Q}\1_{\{0\}}(x)+\widetilde K_j(x),
 \qquad
 \widetilde K_j(x)=\frac{\epsilon_Q}{q^m-\epsilon_Q}
       \sum_{s\ne0}\chi\left(-js-\frac{Q(x)}{4s}\right).
\end{equation}
\end{lemma}

\begin{proof}
Insert
$\1_{S_j}(\xi)=q^{-1}\sum_s\chi(s(Q(\xi)-j))$.
Orthogonality gives $q^{2m-1}\1_{\{0\}}(x)$ for $s=0$, while
\eqref{eq:even-gauss-transform} evaluates the other terms.
At $x=0$, their scalar sum is $-1$ because $j\ne0$.
This gives $|S_j|$, and division by $|S_j|$ gives the normalized formula.
\end{proof}

\begin{lemma}\label{lem:kernel-linfty}
Uniformly in $q$, $m\ge1$, the nondegenerate form $Q$, and $j\ne0$,
\[
 \|\widetilde K_j\|_\infty\lesssim q^{-m+1/2},\qquad
 \left|\sum_{\xi\in S_j}\chi(B(x,\xi))\right|
 \le2q^{m-1/2}\quad(x\ne0).
\]
\end{lemma}

\begin{proof}
The scalar sum in \eqref{eq:unnormalized-sphere-kernel} equals $-1$
when $Q(x)=0$ and has modulus at most $2\sqrt q$ otherwise, by
Lemma~\ref{lem:twisted-kloosterman}. Since
$q^m-\epsilon_Q\ge(2/3)q^m$, both bounds follow.
\end{proof}

The following lemma collects some elementary restriction bounds.

\begin{lemma}
\label{lem:global-auxiliary-bounds}
For every $f:V\to\C$ and every $j\ne0$,
\[
 \|\widehat f\|_{L^2(S_j,d\sigma_j)}
 \lesssim\|f\|_2+q^{-m/2+1/4}\|f\|_1,
 \qquad
 \|\widehat f\|_{L^2(S_j,d\sigma_j)}
 \le\frac{q^m}{\sqrt{|S_j|}}\|f\|_2
 \lesssim q^{1/2}\|f\|_2.
\]
The implicit constants are absolute.
\end{lemma}

\begin{proof}
Expand the square of the restriction norm and apply
\eqref{eq:normalized-sphere-kernel}:
\[
 \|\widehat f\|_{L^2(S_j,d\sigma_j)}^2
 \le\frac{q^m}{q^m-\epsilon_Q}\|f\|_2^2
       +\|\widetilde K_j\|_\infty\|f\|_1^2
 \lesssim\|f\|_2^2+q^{-m+1/2}\|f\|_1^2.
\]
Taking square roots proves the first bound. Plancherel gives
$\sum_{\xi\in S_j}|\widehat f(\xi)|^2\le q^{2m}\|f\|_2^2$,
which proves the second.
\end{proof}

We use the following uniform $\varepsilon$-removal lemma.

\begin{lemma} \label{lem:epsilon-removal}
Fix $m\ge1$ and $2\le b<\infty$, and let $\mathcal Q$ be a family
of odd prime powers. Suppose that, for every
$\varepsilon>0$, there is $C_\varepsilon$ such that
\[
 \|(f\,d\sigma_j)^\vee\|_{L^b(V)}
 \le C_\varepsilon q^\varepsilon
       \|f\|_{L^2(S_j,d\sigma_j)}
\]
for every $q\in\mathcal Q$, every nondegenerate $Q$ on $V$,
every $j\ne0$, and every complex function $f$ on $S_j$.
Then, for every $r>b$,
\[
 \|(f\,d\sigma_j)^\vee\|_{L^r(V)}
 \le C_r\|f\|_{L^2(S_j,d\sigma_j)}
\]
uniformly in $q\in\mathcal Q$, $Q$, $j$, and $f$. The constant depends
only on $m$, $b$, $r$, and the constants in the hypothesis.
\end{lemma}

\begin{proof}
Lemma~\ref{lem:kernel-linfty} gives
$|(d\sigma_j)^\vee(x)|\le Cq^{-m+1/2}\le Cq^{-1/2}$ for $x\ne0$.
For sufficiently large $q$, this is at most $q^{-1/4}$.
Thus, the surfaces have a common positive Fourier dimension.
The $\varepsilon$-removal lemma of \cite[Lemma~16]{european}, with source
exponent two and target exponent $b$, gives the conclusion for
$r>b$. Its constants are uniform for this fixed decay bound and
the constants $C_\varepsilon$; the argument applies along the
prescribed family $\mathcal Q$. The bounded remaining field sizes
are covered by Cauchy--Schwarz:
\[
 \|(f\,d\sigma_j)^\vee\|_{L^r(V)}
 \le q^{2m/r}\|f\|_{L^2(S_j,d\sigma_j)}.
\]
\end{proof}

\subsection{A norm bound for Hecke operators}
\label{sec:weighted-hecke-operator-bound}

Let $K=\F_q(X)$ be the rational function field over $\F_q$, and let
$\mathbb A_K$ be its adele ring. Write $G=\mathrm{GL}_2$ for the
general linear algebraic group over $K$, and $Z$ for its center. Thus,
$G(K)=\mathrm{GL}_2(K)$ and
$G(\mathbb A_K)=\mathrm{GL}_2(\mathbb A_K)$. A unitary central character
$\omega$ is a continuous character of $Z(K)\backslash Z(\mathbb A_K)$.
The scalar automorphic space $\mathcal H_\omega$ consists of measurable
functions satisfying
\[
 F(\gamma z g)=\omega(z)F(g),\qquad
 \|F\|_{\mathcal H_\omega}^2
 =\int_{Z(\mathbb A_K)G(K)\backslash G(\mathbb A_K)}|F(g)|^2\,dg<\infty.
\]
Here $\gamma\in G(K)$ and $z\in Z(\mathbb A_K)$. The absolute value is
well defined on the central quotient because $\omega$ is unitary.

Let $U=\prod_w U_w$ be a compact open level subgroup and let
$\tau:U\to\mathrm U(\mathcal V)$ be a continuous finite-dimensional unitary representation.
We impose the central compatibility condition
\begin{equation}\label{eq:typed-central-compatibility}
 \tau(z)=\omega(z)^{-1}I_{\mathcal V}\qquad(z\in U\cap Z(\mathbb A_K)).
\end{equation}
The space with coefficient representation $\tau$ is
\[
 \mathcal H_\omega(U,\tau)
   =(\mathcal H_\omega\otimes \mathcal V)^{R\otimes\tau}(U),
 \qquad (R(k)F)(g)=F(gk).
\]
Equivalently, it consists of $\mathcal V$-valued functions with the convention
\begin{equation}\label{eq:typed-automorphic-convention}
 \Phi(\gamma z g k)=\omega(z)\tau(k)^{-1}\Phi(g),
 \qquad\|\Phi\|^2=\int_{Z(\mathbb A_K)G(K)\backslash G(\mathbb A_K)}\|\Phi(g)\|_{\mathcal V}^2\,dg.
\end{equation}
Condition~\eqref{eq:typed-central-compatibility} makes the two
transformation rules agree on central elements of $U$.

\begin{proposition}\label{prop:weighted-hecke-bound}
Let $v$ be a rational place. Suppose that
$U_v=\mathrm{GL}_2(\mathcal O_v)$, that $\tau$ is trivial on $U_v$,
and that $\mathcal H_\omega(U,\tau)$ has zero projection onto every
residual determinant representation. Normalize local Haar measure by
$\operatorname{vol}(U_v)=1$. Then the degree-one Hecke operator
\[
 H_v=\1_{U_v\operatorname{diag}(\varpi_v,1)U_v}
\]
has norm at most $2\sqrt q$ on $\mathcal H_\omega(U,\tau)$.
The residue field at a rational place is $\F_q$, so the double coset
is the union of $q+1$ right cosets, each with coefficient one.
The bound is independent of the dimension and conductor of $\tau$.
It also holds for every compression by isometric coordinate embeddings.
\end{proposition}

\begin{proof}
Drinfeld's theorem \cite{Drinfeld} implies that every unramified local
component of a cuspidal representation of $\mathrm{GL}_2(\mathbb A_K)$
is tempered. No assumption is made on its ramification at the other
places. Thus, at $v$ its normalized Satake parameters $\alpha,\beta$
have modulus one. The eigenvalue of $H_v$ is
$\sqrt q(\alpha+\beta)$ and has modulus at most $2\sqrt q$.
To pass from the usual finite-order-central-character formulation to
an arbitrary unitary $\omega$, twist by a unitary character of the
idele degree. The restriction of $\omega$ to the compact degree-zero
idele class group has finite order. One can choose the twist $\nu$ so
that $\omega\nu^2$ is trivial on a degree-one generator. The twisted
central character then has finite order. Such a twist is unramified
everywhere and does not change the moduli of the Satake parameters.

The continuous spectrum is a direct integral of normalized parabolic
inductions from unitary characters $\mu_1,\mu_2$;
see \cite[Section~1.2]{Flicker}. When the local induction at $v$ has
$U_v$-fixed vectors, its Hecke eigenvalue is
\[
 \sqrt q\bigl(\mu_{1,v}(\varpi_v)+\mu_{2,v}(\varpi_v)\bigr),
\]
whose modulus is at most $2\sqrt q$. The residual spectrum consists
of determinant-character lines; see
\cite[Section~4, Proposition~4.29]{Flicker}. It is excluded by hypothesis.

The operator $H_v$ is normal. Indeed, its inverse double coset is the
original double coset multiplied by the central scalar
$\varpi_v^{-1}I$. Therefore, $H_v^*$ is a scalar multiple of
$H_v$ on $\mathcal H_\omega$, and the scalar has modulus one. The spectral bounds give the claimed
operator norm on the orthogonal complement
$\mathcal H_{\omega,\mathrm{nr}}$ of the residual spectrum.
Residual spectral projections commute with the right level action.
Hence, the hypothesis places the space with coefficients inside
$\mathcal H_{\omega,\mathrm{nr}}\otimes \mathcal V$. The operator there is
$H_v\otimes I_{\mathcal V}$. Tensoring with $I_{\mathcal V}$, passing to the closed invariant
space, and compressing by isometries do not increase its norm.
\end{proof}

At a full flag, a diagonal character $(\omega_0,1)$ with $\omega_0\ne1$
excludes every determinant representation: its two diagonal characters
would have to be equal. A nontrivial character on an element of
determinant one in a residue unipotent subgroup also excludes them.
We use both criteria below. The local coefficient characters determine
the inverse of the automorphic central character through
\eqref{eq:typed-central-compatibility}. When extending these characters
along the degree factor, we may choose a finite-order degree character,
for example the trivial one.

We next specify the coordinate weights. Put
$\overline G=Z(\mathbb A_K)\backslash G(\mathbb A_K)$, and let
$\overline\Gamma$ and $\overline U$ be the images of $G(K)$ and $U$.
For a representative $g_a$ of
$\overline\Gamma\backslash\overline G/\overline U$, define
\[
 \Gamma_a=\overline\Gamma\cap
                  \overline g_a\overline U\overline g_a^{-1},
 \qquad w_a=\frac{\operatorname{vol}(\overline U)}{|\Gamma_a|}.
\]
The group $\Gamma_a$ is finite. The allowed coefficient fiber is
\begin{equation}\label{eq:typed-stabilizer-fiber}
\mathcal V_a=\{v\in \mathcal V:\omega(z)\tau(k)v=v\text{ whenever }g_a k=\gamma z g_a\},
\end{equation}
where $\gamma\in G(K)$, $z\in Z(\mathbb A_K)$, and $k\in U$.
The scalar $\omega(z)$ is independent of the chosen rational central
factor, and \eqref{eq:typed-central-compatibility} removes the ambiguity
from $U\cap Z(\mathbb A_K)$. Thus, the fiber condition is well defined.

\Needspace{12\baselineskip}
\begin{lemma} \label{lem:typed-hecke-compression}
Under the hypotheses of Proposition~\ref{prop:weighted-hecke-bound},
evaluation on the representatives $g_a$ identifies the space with
coefficients isometrically with
\[
 \left\{(v_a):v_a\in \mathcal V_a,\quad\sum_a w_a\|v_a\|_{\mathcal V}^2<\infty\right\}.
\]
Let $h_1,\ldots,h_{q+1}$ represent the right $U_v$-cosets of
$U_v\operatorname{diag}(\varpi_v,1)U_v$, embedded at the place $v$.
For finite sets of source and target classes, the matrix in orthonormal
fiber coordinates has entries
\begin{equation}\label{eq:typed-compression-matrix}
H_{ab}=\sqrt{\frac{w_a}{w_b}}\,
 P_{\mathcal V_a}\sum_{\substack{i:\,[g_a h_i]=[g_b]}}\omega(z_i)\tau(k_i)^{-1}\big|_{\mathcal V_b},
 \qquad g_a h_i=\gamma_i z_i g_b k_i.
\end{equation}
Here $P_{\mathcal V_a}$ is orthogonal projection onto $\mathcal V_a$, and every right
coset is counted once. The matrix has norm at most $2\sqrt q$.
If the source and target weights are equal and their stabilizers act
trivially on $\mathcal V$, the weight ratio and the fiber projections disappear.
\end{lemma}

\begin{proof}
The image of $g_a\overline U$ in
$\overline\Gamma\backslash\overline G$ has measure
$\operatorname{vol}(\overline U)/|\Gamma_a|$. Its stabilizer is the
finite group $\Gamma_a$, and the transformation law
\eqref{eq:typed-automorphic-convention} makes the norm of a section
constant on this orbit. It also shows that a value at $g_a$ extends
consistently over the orbit exactly when it lies in $\mathcal V_a$.
This proves the asserted Hilbert-space identification.

The isometric embedding $J_a:\mathcal V_a\to\mathcal H_\omega(U,\tau)$
has value $w_a^{-1/2}v$ at $g_a$ and vanishes on the other classes.
Since the type is trivial at $v$, right convolution gives
$(H_v\Phi)(g_a)=\sum_i\Phi(g_a h_i)$.
For $g_a h_i=\gamma_i z_i g_b k_i$, the summand is
$\omega(z_i)\tau(k_i)^{-1}\Phi(g_b)$.
Applying $J_a^*$ and $J_b$ gives
\eqref{eq:typed-compression-matrix}. Its norm is bounded by the norm
of $H_v$, as claimed.
\end{proof}

For the bundle descriptions used below, the base curve is
$\mathbb P^1$. Quotienting by the adelic center identifies bundles
that differ by tensoring with a line bundle. A rank-two bundle cannot
be isomorphic to its twist by a line bundle of nonzero degree; a
line bundle of degree zero on $\mathbb P^1$ is trivial. If the level
contains all scalar units, this identifies $\Gamma_a$ with
$\operatorname{Aut}(E_a)/\F_q^\times$. Consequently,
\[
 w_a=\operatorname{vol}(\overline U)
                \frac{q-1}{|\operatorname{Aut}(E_a)|}.
\]
Thus, the reciprocal-automorphism weights differ from the
central-quotient weights by one common positive factor. In particular,
two objects with scalar automorphism group $\F_q^\times$ have equal
weights. Their entire coefficient spaces may be used precisely when
the combined stabilizer action in \eqref{eq:typed-stabilizer-fiber}
is trivial.

\subsection{Finite group representations}

We use the unitary linear Weil representation of
$\mathrm{Sp}_4(\F)$ on $L^2(\F^2)$, as in
\cite[Theorem~1.3.2]{GH}. In
Section~\ref{sec:exact-weil-root-normalization}, we specify its action
on generators and derive its Gaussian kernel. These formulas fix
the character and sign conventions. Its restriction to any symplectic
subgroup is again a unitary linear representation. This also applies
to the subgroups arising from a nonreduced quadratic algebra below.

\begin{lemma}\label{lem:sl2-regular-class-bound}
Let $q$ be odd, let $\mathcal C$ be a regular semisimple conjugacy class
in $\mathrm{SL}_2(\F)$, and let $\chi_\pi$ be a nontrivial irreducible
character. For $g\in\mathcal C$,
\[
 \frac{|\mathcal C|\,|\chi_\pi(g)|}{\chi_\pi(1)}\le2q.
\]
\end{lemma}

\begin{proof}
A split regular class has size $q(q+1)$, and a nonsplit regular
class has size $q(q-1)$. The character table
\cite[Table~1]{BFL} gives modulus at most two for the principal-series
values on a split class, with degree $q+1$, and for the cuspidal
values on a nonsplit class, with degree $q-1$. The exceptional irreducible characters
have modulus one and half these degrees. The principal-series values on nonsplit regular classes and the cuspidal values on split regular classes vanish.
The Steinberg character has modulus one on these classes and degree
$q$. These bounds prove the assertion in every case.
\end{proof}

\section{A convolution identity and reduction to matrices}
\label{sec:exact-product-reduction}

Let $Q$ be a nondegenerate quadratic form on $V=\F^{2m}$,
and let $j\ne0$. Recall that $S_j=\{x:Q(x)=j\}$. For $t\in V$,
define the real symmetric matrix on $L^2(S_j)$ by
\begin{equation}\label{eq:centered-shell-matrix}
 M_t(x,y)=\1_{S_j}(x+y-t)-\frac{|S_j|}{q^{2m}},
 \qquad x,y\in S_j.
\end{equation}
Here we subtract the density of $S_j$ in $V$.
The following identity relates these matrices to convolution.

\begin{lemma}\label{lem:product-block-identity}
Let $f:V\to\C$ be supported on $S_j$, and let
$g:V\to\C$ be arbitrary. Put
\[
 a_t(x)=f(x)\overline{g(x-t)},\qquad x\in S_j.
\]
Then
\begin{equation}\label{eq:product-block-expansion}
 \|\1_{S_j}(f*g)\|_2^2-\frac{|S_j|}{q^{2m}}\|f*g\|_2^2
   =\sum_{t\in V}\langle M_ta_t,a_t\rangle,
\end{equation}
and
\begin{equation}\label{eq:product-vector-mass}
 \sum_{t\in V}\|a_t\|_2^2=\|f\|_2^2\|g\|_2^2.
\end{equation}
\end{lemma}

\begin{proof}
Since $\operatorname{supp}f\subseteq S_j$, expanding the square gives
\[
 \|\1_{S_j}(f*g)\|_2^2
 =\sum_z\1_{S_j}(z)\sum_{x,y\in S_j}
       f(x)g(z-x)\overline{f(y)g(z-y)}.
\]
For fixed $x,y$, make the bijective change of variables $t=x+y-z$.
The summand becomes
\[
 \1_{S_j}(x+y-t)
 f(x)g(y-t)\overline{f(y)g(x-t)}
 =\1_{S_j}(x+y-t)a_t(x)\overline{a_t(y)}.
\]
The same expansion without $\1_{S_j}(z)$ gives $\|f*g\|_2^2$.
Subtracting $|S_j|/q^{2m}$ times that identity proves
\eqref{eq:product-block-expansion}. Since $M_t$ is symmetric, the
resulting quadratic form is $\langle M_ta_t,a_t\rangle$ with our
inner-product convention.
Finally,
\[
 \sum_t\|a_t\|_2^2
 =\sum_{x\in S_j}|f(x)|^2\sum_t|g(x-t)|^2
 =\|f\|_2^2\|g\|_2^2.\qedhere
\]
\end{proof}

In the next section, we prove
\begin{equation}\label{eq:generic-shell-bound}
 \|M_t\|_{2\to2}\le C_mq^{m-1}
 \qquad\text{whenever }Q(t)\ne j;
\end{equation}
see Theorem~\ref{thm:completed-generic-blocks}. The constant $C_m$ depends only on $m$. The bound holds on
all of $L^2(S_j)$, for both Witt types and every odd characteristic.

\Needspace{8\baselineskip}
The estimate \eqref{eq:generic-shell-bound} applies only when
$Q(t)\ne j$. When $t\in S_j$, write $x=t+u$ and $y=t+v$.
Since $x,y,t\in S_j$,
\begin{equation}\label{eq:tangent-shell-polarity}
 Q(u)+2B(t,u)=Q(v)+2B(t,v)=0,
 \qquad Q(x+y-t)=j+2B(u,v).
\end{equation}
For nonzero $u,v$, the condition $B(u,v)=0$ is a projective
orthogonality relation.
In Section~\ref{sec:fourth-moments-sharp-extension}, we use it to
treat the terms with $Q(u),Q(v)\ne0$. A convolution estimate between
the sphere and the cone $\{Q=0\}$ treats the remaining terms, using
the vectors $a_t$ from Lemma~\ref{lem:product-block-identity}.
We then combine these estimates with
\eqref{eq:product-block-expansion}--\eqref{eq:product-vector-mass}
to prove a refined fourth-moment bound and the strong endpoint estimate.

\section{Hecke operators and matrix estimates}
\label{sec:hecke-matrix-estimates}

We prove the shifted-sphere estimate~\eqref{eq:generic-shell-bound}
in every even dimension. The arithmetic operators in the proof act
on functions of one or two variables over $\F_q$; their
bounds do not depend on the ambient dimension. We first prove the
representation-valued Hecke estimate and specify the finite Weil
normalization. We then treat the rank-one matrices, separating
characteristic three from the formulas that divide by $3$.
The final subsection applies these estimates after Fourier
transformation in Witt coordinates and completes the induction on $m$.
All finite-dimensional norms use counting measure, and each root
in a finite field is counted once.

\Needspace{14\baselineskip}
\subsection{A Hecke bound for a weighted operator sum}
\label{sec:arithmetic-whittaker-hecke}

\begin{theorem}\label{thm:arithmetic-whittaker-curve}
Let $q$ be odd, let $\mathcal K$ be a two-dimensional commutative
$\mathbb F_q$-algebra, and let $a\in\mathcal K$ satisfy
$a^2+a+b=0$ for $b\in\mathbb F_q^\times$. Suppose that
$a$ and $a^2-1$ are units and that $1,a^2$ are linearly independent.
For $t\in\mathbb F_q^\times$ put
\[
 h_t=\begin{pmatrix}a&t\\(a^2-1)/t&a\end{pmatrix}.
\]
Let $\rho$ be any finite-dimensional unitary representation of
$\mathrm{SL}_2(\mathcal K)$, and let $c\in\mathbb F_q$.
Then
\[
 \left\|\sum_{t\ne0}\eta(t)\chi(ct)\rho(h_t)\right\|
 \le 2\sqrt q.
\]
Here $\eta$ is the quadratic character. The assertion includes the
split algebra, a quadratic field, and the dual-number algebra.
\end{theorem}

\begin{proof}
We construct the automorphic model and check its coefficient
space and measure. The base curve is $\mathbb P^1_X$.
Since $a^2$ generates $\mathcal K$, the element $a^{-2}$ does also,
and evaluation $X\mapsto a^{-2}$ identifies the degree-two algebra
\[
 \mathbb F_q[X]/(P(X))\simeq\mathcal K,
 \qquad
 P(X)=X^2-\frac{1-2b}{b^2}X+\frac1{b^2}.
\]
The divisor $P$ is disjoint from $0,1,\infty$, because $a$ and
$a^2-1$ are units. It may consist of two rational points, a degree-two
point, or a doubled rational point. In the last case all the level
conditions below mean reduction modulo the square of its local
parameter.

First suppose that $\rho(-I)=\epsilon I$, $\epsilon\in\{1,-1\}$.
The general case is the orthogonal sum of these two invariant spaces.
Choose a unitary character $\vartheta_{\mathcal K}$ of $\mathcal K^\times$ with
$\vartheta_{\mathcal K}(-1)=\epsilon$. Such a character exists because the character
group separates the nonidentity element $-1$. Write
\[
 G^+=\{g\in\mathrm{GL}_2(\mathcal K):\det g\in
                         (\mathcal K^\times)^2\}.
\]
There is a representation of $G^+$ defined by
\[
 r(g)=\vartheta_{\mathcal K}(\delta)\rho(g/\delta),\qquad \delta^2=\det g.
\]
This is independent of the choice of square root: changing
$\delta$ to $-\delta$ changes both factors by $\epsilon$.
It is multiplicative, since scalar matrices commute with every matrix.
For the split algebra the equation $\delta^2=\det g$ can have four
solutions, rather than two. In this case decompose further under the
central involutions of $\mathrm{SL}_2(\mathcal K)$, and choose $\vartheta_{\mathcal K}$
whose restriction to every element $u$ with $u^2=1$ equals the scalar
$\rho(uI)$. A character on this subgroup extends to
$\mathcal K^\times$, so the same formula is well-defined. This construction applies to all three algebra types.

Precomposing $\rho$ with a fixed $\mathrm{GL}_2(\mathcal K)$
conjugation preserves these properties. We now specify that
conjugation. Put
\[
 D_0=\operatorname{diag}(-a/(a^2-1),1),\qquad
 J_0=\begin{pmatrix}0&1\\a^2-1&0\end{pmatrix},\qquad C_0=J_0D_0.
\]
Use in the definition of $r$ the representation
$g\mapsto\rho(C_0gC_0^{-1})$.
At the divisor $P$, take the inverse image of $G^+$ as compact level,
with coefficient representation $r$. Together with the types at $1$
and infinity specified below, this defines the representation $\tau$ of the compact level subgroup. We use
$\Phi(\gamma z gk)=\omega(z)\tau(k)^{-1}\Phi(g)$ throughout.

At $X=1$ take the compact subgroup whose residue lies in
\[
 ZU_-=
 \left\{s\begin{pmatrix}1&0\\u&1\end{pmatrix}:
           s\in\mathbb F_q^\times,\ u\in\mathbb F_q\right\}.
\]
Give it the character
\[
 s\begin{pmatrix}1&0\\u&1\end{pmatrix}
 \longmapsto \mu(s)\chi(-cu),
 \qquad \mu=(\vartheta_{\mathcal K}|_{\mathbb F_q^\times}\eta)^{-1}.
\]
At infinity take an ordinary flag, with $\eta$ applied to the scalar
on its line. These three local coefficient characters have product
one on $\mathbb F_q^\times$. They therefore extend to a finite-order
idele class character: equivalently, they define a character of the
degree-zero ray class group, whose presentation on $\mathbb P^1$ is
the product of the residue unit groups modulo the diagonal constants.
Choose the trivial character on the degree factor, and let $\omega$
be the inverse of this extension. Then
$\tau(z)=\omega(z)^{-1}I$ on the compact center, as required by
Lemma~\ref{lem:typed-hecke-compression}. All remaining places have
maximal compact level.

At infinity the local type on the diagonal torus is $(\eta,1)$.
A determinant representation restricts there as $(\nu,\nu)$ for some
character $\nu$, so cannot contain this type: varying the two diagonal
entries independently would force both $\nu=\eta$ and $\nu=1$.
Thus, the selected automorphic $L^2$ space has no residual
determinant representations, including when $c=0$.
When $c\ne0$, the nontrivial unipotent character at $1$ gives a
second exclusion. Proposition~\ref{prop:weighted-hecke-bound} gives $\|H_0\|\le2\sqrt q$ for the unnormalized degree-one Hecke operator at $X=0$.
The maximal compact subgroup at $0$ has Haar volume one, so this
operator sums its $q+1$ right cosets with weight one. The finite index
of $G^+$ at $P$ changes neither this normalization nor the bound.

We next identify a compression of this operator. The source object is
$E=\mathcal O^2$, with the standard $ZU_-$ reduction at $1$, the standard
level at $P$, and the infinity flag $e_1$. The target object is
$F=\mathcal O(1)\oplus\mathcal O$, with the same standard reductions
at the finite places and infinity flag $e_2$. Both have automorphism
group exactly $\mathbb F_q^\times$. Indeed, the source automorphisms
fixing its infinity flag are upper triangular constant matrices,
whose intersection with $ZU_-$ is the center. A target automorphism
has the form
\[
 \begin{pmatrix}\alpha&\beta+\gamma X\\0&\delta\end{pmatrix}.
\]
In the infinity frame this has value
$\left(\begin{smallmatrix}\alpha&\gamma\\0&\delta\end{smallmatrix}\right)$.
Fixing the infinity flag $e_2$ forces $\gamma=0$; its reduction at $1$
belongs to $ZU_-$ only when $\beta=0$ and $\alpha=\delta$.
The coefficient action of each central constant is trivial after
multiplying the three local factors, by the choice of $\mu$.
Thus, the full representation space of $\rho$ occurs at both objects.
Their coordinate weights are equal to a common positive multiple of
$(q-1)^{-1}$, and the normalized embeddings $i_E,i_F$ of
Lemma~\ref{lem:typed-hecke-compression} are isometric.
Here quotienting by the adelic center introduces no relative factor.
Indeed, an isomorphism from either bundle to its tensor product with
a line bundle $L$ forces $2\deg L=0$. On $\mathbb P^1$, such $L$
is trivial, and all our compact levels contain the scalar unit groups.
The degree-zero and degree-one objects also remain distinct modulo
tensoring by line bundles, which changes their degree by an even
integer. This observation applies to the two bundle families used
later in the section as well. No projection to a proper subspace of
the coefficient fiber is needed.

For $z\in\mathbb F_q^\times$ consider
\[
 A_z(X)=\begin{pmatrix}1&z(X-1)\\-1/z&1\end{pmatrix},
 \qquad \det A_z(X)=X.
\]
This is an elementary modification $E\hookrightarrow F$ at $0$.
Its kernel at $0$ is $\langle(z,1)\rangle$. In the target infinity
frame the map has value
\[
 \left.\operatorname{diag}(X^{-1},1)A_z(X)\right|_{\infty}
 =\begin{pmatrix}0&z\\-1/z&1\end{pmatrix}.
\]
These are exactly the modifications between the two specified objects.
For completeness, a general bundle map has the form
\[
 \begin{pmatrix}\alpha+\beta X&\gamma+\delta X\\
                  e&f\end{pmatrix}.
\]
Preservation of the two infinity flags forces $\beta=0$ and $e\ne0$.
Membership of its value at $1$ in $ZU_-$ forces
$f=\alpha\ne0$ and $\gamma=-\delta$.
Its determinant vanishes at $0$ precisely when
$\delta e=-\alpha^2$. Dividing by the common central scalar $\alpha$
then gives exactly $A_z$, with $z=\delta/\alpha\ne0$.
The two other kernel lines at $0$ give targets outside the selected
chart. This proves completeness.

We specify representatives to identify the right-coset action. We use
the following matrix $w_0$ throughout the section:
\[
 w_0=\begin{pmatrix}0&1\\1&0\end{pmatrix},\qquad
 g_E=I,\qquad g_{F,\infty}=\operatorname{diag}(X,1)w_0,
\]
with $g_{F,v}=I$ at every finite place. Let $a_z$ be supported at
$0$, with component $A_z(X)$ there, and put $\gamma_z=A_z$ as a
rational matrix. Then
\[
 g_Fa_z=\gamma_zg_Ek_z,\qquad k_{z,0}=I,\qquad
 k_{z,v}^{-1}=g_{F,v}^{-1}A_zg_{E,v}\quad(v\ne0).
\]
The matrices $k_{z,v}$ belong to the stated compact levels. The
right cosets $a_zU_0$ are distinct, since the image lines of
$A_z(0)$ are $\langle(z,-1)\rangle$. Thus, every listed modification
occurs once in the Hecke action. Its coefficient is
$\tau(k_z)^{-1}$, namely the product of the coefficient actions of the framed maps
$g_{F,v}^{-1}A_zg_{E,v}$.

We now compute these coefficient factors. At $1$,
\[
 A_z(1)=\begin{pmatrix}1&0\\-1/z&1\end{pmatrix}
\]
gives $\chi(c/z)$. At infinity the framed map has reduction
\[
 w_0^{-1}\left.\operatorname{diag}(X^{-1},1)A_z(X)\right|_{\infty}
 =\begin{pmatrix}-1/z&1\\0&z\end{pmatrix},
\]
so its factor on the flagged line is $\eta(-1/z)$. At $P$, its determinant is
$a^{-2}$, so take the fixed square root $\delta=a^{-1}$. Direct
multiplication gives
\[
 D_0\,aA_z(a^{-2})\,D_0^{-1}=h_z,
 \qquad J_0h_zJ_0^{-1}=h_{1/z}.
\]
Thus, the $P$ coefficient is the constant scalar of modulus one
$\vartheta_{\mathcal K}(a^{-1})$ times $\rho(h_{1/z})$.
Lemma~\ref{lem:typed-hecke-compression} therefore gives the exact
identity
\begin{equation}\label{eq:weighted-hecke-exact-compression}
 i_F^*H_0i_E
 =\eta(-1)\vartheta_{\mathcal K}(a^{-1})
   \sum_{t\ne0}\eta(t)\chi(ct)\rho(h_t).
\end{equation}
This holds on each eigenspace of the central involutions used to choose
$\vartheta_{\mathcal K}$. The scalar has modulus one, so the norm
bound for $H_0$ proves the theorem on each summand and hence on their
orthogonal sum.

\end{proof}

The following normalization identifies the oscillatory matrices
that occur in the higher-dimensional Fourier reduction with this
operator sum.

\subsection{Gaussian normalization of the Weil operators}
\label{sec:exact-weil-root-normalization}

We use the finite Weil representation of $\mathrm{Sp}_4(\F_q)$ to
represent quadratic kernels on $\F_q^2$. The symplectic space has dimension four, independently of the dimension of the quadratic space in the restriction problem.
The normalization below applies in every odd characteristic and
the embedding below also applies to a nonreduced quadratic algebra. Recall that $G_\eta=\sum_{z\in\F_q}\chi(z^2)$.

\begin{lemma}
\label{lem:weil-generator-normalization}
Let $q$ be odd, and give $\F^2\oplus\F^2$ the symplectic form
\[
 \omega_{\mathrm{sp}}((a,b),(a',b'))
       =a^{\top}b'-b^{\top}a'.
\]
In the Schr\"odinger model associated with the additive character
$\chi$, the canonical linear Weil representation $\rho_{\mathrm W}$
has the following action. For a symmetric matrix $C$ and an invertible
matrix $A$, put
\[
 \ell(C)=\begin{pmatrix}I&0\\C&I\end{pmatrix},\qquad
 M(A)=\begin{pmatrix}A&0\\0&A^{-\top}\end{pmatrix},\qquad
 w=\begin{pmatrix}0&I\\-I&0\end{pmatrix}.
\]
Then
\begin{equation}\label{eq:weil-generator-normalization}
 \begin{aligned}
 \rho_{\mathrm W}(\ell(C))f(x)
   &=\chi(-x^{\top}Cx/2)f(x),\\
 \rho_{\mathrm W}(M(A))f(x)
   &=\eta(\det A)f(A^{-1}x),\\
 \rho_{\mathrm W}(w)f(x)   &=\frac{\eta(-1)}q\sum_{y\in\F^2}\chi(x^{\top}y)f(y).
 \end{aligned}
\end{equation}
Consequently, if a symplectic matrix has invertible upper-right block
$\beta$, its kernel is
\begin{equation}\label{eq:weil-open-cell-kernel}
 \rho_{\mathrm W}
 \begin{pmatrix}\alpha&\beta\\\gamma&\delta\end{pmatrix}(x,y)
 =\frac{\eta(-\det\beta)}q
 \chi\left(-\frac12\left[
 x^{\top}\delta\beta^{-1}x
 -2y^{\top}\beta^{-1}x
 +y^{\top}\beta^{-1}\alpha y\right]\right).
\end{equation}
\end{lemma}

\begin{proof}
We make the convention in the Schr\"odinger model explicit. The
Heisenberg group has multiplication
\[
 (a,b,z)(a',b',z')=
 \left(a+a',b+b',z+z'+
       \tfrac12(a^{\top}b'-b^{\top}a')\right)
\]
and representation
\[
 (\pi(a,b,z)f)(x)
   =\chi(z+b^{\top}x+a^{\top}b/2)f(x+a).
\]
It is irreducible with central character $\chi$. The canonical linear
representation of \cite[Theorem~1.3.2]{GH} satisfies the Egorov identity
\[
 \rho_{\mathrm W}(g)\pi(a,b,z)\rho_{\mathrm W}(g)^{-1}
       =\pi(g(a,b),z).
\]
Define the unitary operators
\[
 L_Cf(x)=\chi(-x^{\top}Cx/2)f(x),\quad
 D_Af(x)=f(A^{-1}x),\quad
 \mathcal F_+f(x)=q^{-1}\sum_y\chi(x^{\top}y)f(y).
\]
Direct substitution shows that they satisfy this identity for
$\ell(C)$, $M(A)$, and $w$, respectively. Indeed, their conjugations
send $(a,b)$ to $(a,b+Ca)$, $(Aa,A^{-\top}b)$, and $(b,-a)$.
By irreducibility, each differs from the corresponding Weil operator
by a scalar. We determine these scalars from the character formula.

For $g-I$ invertible, \cite[Theorem~2.2.1]{GH} gives, in dimension four,
\[
 \operatorname{tr}\rho_{\mathrm W}(g)
   =\frac{G_\eta^4}{q^2}
      \eta\bigl(\det(\mathfrak c(g)+I)\bigr)
   =\eta\bigl(\det(g-I)\bigr),
 \qquad \mathfrak c(g)=(g+I)(g-I)^{-1}.
\]
The last equality follows from $G_\eta^4=q^2$,
$\mathfrak c(g)+I=2g(g-I)^{-1}$, and $\det g=1$.
Since $\det(w-I)=4$, the trace of $\rho_{\mathrm W}(w)$ is one.
On the other hand,
\[
 \operatorname{tr}\mathcal F_+
    =q^{-1}\sum_{x\in\F^2}\chi(x^{\top}x)
    =G_\eta^2/q=\eta(-1).
\]
Thus, $\rho_{\mathrm W}(w)=\eta(-1)\mathcal F_+$.

If $A-I$ is invertible, then $D_A$ fixes only the basis vector indexed by zero,
so its trace is one. Moreover,
\[
 \det(M(A)-I)=\det(A)^{-1}\det(A-I)^2.
\]
The character formula therefore gives
$\rho_{\mathrm W}(M(A))=\eta(\det A)D_A$.
To obtain the formula for every $A$, write
$\rho_{\mathrm W}(M(A))=\xi(A)D_A$. The Egorov identity and
irreducibility show that $\xi(A)$ is a scalar, and multiplicativity
shows that $\xi$ is a character of $\mathrm{GL}_2(\F)$.
Conjugation by $\operatorname{diag}(-1,1)$ sends each elementary
upper or lower triangular unipotent matrix to its inverse. Its image
under $\xi$ therefore has order dividing two. The same image has odd
order, since the characteristic is odd, so it equals one. Elementary
unipotent matrices generate $\mathrm{SL}_2(\F)$; hence $\xi(A)$
depends only on $\det A$.
For each $u\in\F^\times$, choose $t\in\F$ with $t\ne1+u$ and put
\[
 A_u=\begin{pmatrix}0&-u\\1&t\end{pmatrix}.
\]
Then $\det A_u=u$ and $\det(A_u-I)=1-t+u\ne0$. The formula already
proved gives $\xi(A_u)=\eta(u)$, and hence
$\xi(A)=\eta(\det A)$ for every $A$. This argument includes $q=3$.

Finally, choose $c\in\F^\times\setminus\{1\}$. For
$g=\ell(C) M(cI)$, the determinant of $g-I$ is
$(c-1)^2(c^{-1}-1)^2$, and hence $\operatorname{tr}\rho_{\mathrm W}(g)=1$.
The operator $L_C D_{cI}$ satisfies the Egorov identity for $g$ and
also has trace one: only the basis vector indexed by zero contributes to the trace, and there the
phase equals one. Therefore,
$\rho_{\mathrm W}(g)=L_C D_{cI}$. Dividing by
$\rho_{\mathrm W}(M(cI))=D_{cI}$ proves the shear formula, including
when $C$ is singular.

For invertible $\beta$, the symplectic identities imply that
$\delta\beta^{-1}$ and $\beta^{-1}\alpha$ are symmetric and give
\[
 \begin{pmatrix}\alpha&\beta\\\gamma&\delta\end{pmatrix}
 =\ell(\delta\beta^{-1})\,M(\beta)\,w\,\ell(\beta^{-1}\alpha).
\]
Multiplying the operators in \eqref{eq:weil-generator-normalization}
in this order proves \eqref{eq:weil-open-cell-kernel}.
\end{proof}
\begin{corollary}
\label{cor:full-root-operator-bound}
Let $D,R$ be invertible symmetric two-by-two matrices over $\F_q$,
put $a=R^{-1}D$, and suppose that $\mathcal K=\F_q[a]$ is
two-dimensional. Assume
\[
 a^2+a+b=0,\qquad b\in\F_q^\times,
 \qquad a^2-I\text{ invertible},\qquad 1,a^2\text{ linearly independent}.
\]
For $x,y\in\F_q^2$, set
$H(x,y)=x^{\mathsf T}Dx+y^{\mathsf T}Dy-2x^{\mathsf T}Ry$.
Then, for every $c\in\F_q$,
\begin{equation}\label{eq:binary-gaussian-operator-bound}
 \left\|\sum_{s\ne0}\eta(s)\chi(c/s)
           [\chi(sH(x,y))]_{x,y}\right\|\le2q^{3/2}.
\end{equation}
In particular, the root matrices
$K_\lambda(x,y)=\sum_{z^2=H(x,y)}\chi(\lambda z)$ satisfy
\[
 \|K_\lambda\|\le2q\quad(\lambda\ne0),
 \qquad \|K_0-J\|\le2q,
\]
where $J$ is the all-ones matrix on $\F_q^2$.
\end{corollary}

\begin{proof}
Every element of $\mathcal K$ is self-adjoint for $D$, since
$Da=DR^{-1}D$ is symmetric. Thus, $\mathrm{SL}_2(\mathcal K)$
acts symplectically on $\F_q^2\oplus\F_q^2$ with pairing
\[
 \omega_D((x,\zeta),(y,\theta))
 =x^{\mathsf T}D\theta-\zeta^{\mathsf T}Dy.
\]
Restrict the Weil representation to this subgroup. No trace form
on $\mathcal K$ is used, so this construction also applies when
$\mathcal K$ is nonreduced. In these coordinates put
\[
 g_s=\begin{pmatrix}a&-a/(2s)\\2s(a^{-1}-a)&a\end{pmatrix}.
\]
On replacing $\zeta$ by $D\zeta$ to obtain the standard symplectic coordinates, its
upper-right block is $-R^{-1}/(2s)$. The formula in
Lemma~\ref{lem:weil-generator-normalization} gives exactly
\begin{equation}\label{eq:exact-binary-weil-kernel}
 \rho_{\mathrm W}(g_s)(x,y)
 =q^{-1}\eta(-\det R)\chi(sH(x,y)).
\end{equation}
Indeed, the two diagonal quadratic blocks in the phase are
$-2sD$, and $\eta(-\det(-R^{-1}/(2s)))=\eta(-\det R)$.
Put
\[
 E=\operatorname{diag}(-a/2,1),\qquad
 h_t=\begin{pmatrix}a&t\\(a^2-1)/t&a\end{pmatrix}.
\]
Then $g_s=Eh_{1/s}E^{-1}$. Conjugation by $E$ is a group
automorphism of $\mathrm{SL}_2(\mathcal K)$, so
$\widetilde\rho(h)=\rho_{\mathrm W}(EhE^{-1})$ is a unitary
representation. Changing variables $t=s^{-1}$ in
\eqref{eq:exact-binary-weil-kernel} and applying
Theorem~\ref{thm:arithmetic-whittaker-curve} gives
\[
 \left\|\sum_{s\ne0}\eta(s)\chi(c/s)
            [\chi(sH(x,y))]_{x,y}\right\|
 =q\left\|\sum_{t\ne0}\eta(t)\chi(ct)\widetilde\rho(h_t)\right\|
 \le2q^{3/2}.
\]
Finally, additive orthogonality and a one-dimensional Gauss sum give
\[
 K_\lambda-\1_{\{\lambda=0\}}J
 =\frac{G_\eta}{q}\sum_{s\ne0}\eta(-s)
       \chi\left(\frac{\lambda^2}{4s}\right)
       [\chi(sH(x,y))]_{x,y}.
\]
Since $|G_\eta|=\sqrt q$, the first bound proves both assertions.
\end{proof}

\subsection{Fourier transform of the rank-one matrix}
\label{sec:rankone-affine-fourier}

Assume first that $\operatorname{char}\mathbb F_q>3$.
The case of a nonzero isotropic translation leads to the matrix
\[
 L_c(H,K)=\sum_{z\ne0}\chi\left(z+
       \frac{c(H^2+HK+K^2-1/3)}z\right),\qquad c\ne0.
\]
We keep the constant $-1/3$ in this definition. The next three
subsections prove $\|L_c\|\le10q$ for every $c\ne0$.

For $\beta\ne0$, the Sali\'e identity \cite{Salie} is
\[
 \sum_{z\ne0}\eta(z)\chi(\alpha z+ \beta/z)=G_\eta\,\eta(\beta)\sum_{v^2=4\alpha\beta}\chi(v),
\]
where each root is counted once. To prove this identity, take the unnormalized Fourier transform in $\alpha$. At a nonzero frequency $u$,
the left side transforms to $q\eta(u)\chi(\beta/u)$. The right side
becomes $G_\eta\,\eta(\beta)\sum_v\chi(v-uv^2/(4\beta))$, which has the same
value by completion of the square and $G_\eta^2=\eta(-1)q$.
Both transforms vanish at $u=0$.

Let $\mathcal F_1(K,Y)=q^{-1/2}\chi(KY)$ be the unitary
one-dimensional Fourier matrix. Completing the square in $K$ gives
\begin{align*}
 (L_c\mathcal F_1)(H,Y)
 &=\frac{G_\eta\,\eta(c)}{\sqrt q}\chi(-HY/2)
       \sum_{z\ne0}\eta(z)\chi(A_Yz+B_H/z),\\
 A_Y&=1-Y^2/(4c),\qquad B_H=c(9H^2-4)/12.
\end{align*}
Consequently, whenever $9H^2-4\ne0$,
\[
 (L_c\mathcal F_1)(H,Y)
 =\sqrt q\,\eta\left(-\frac{9H^2-4}{12}\right)\chi(-HY/2)
       \sum_{v^2=(9H^2-4)(4c-Y^2)/12}\chi(v).
\]
Every entry has modulus at most $2\sqrt q$. On the at most two
rows with $B_H=0$, the original formula is an elementary Gauss sum or
zero, and the sharper entry bound $\sqrt q$ holds. We use these entry bounds below when adding the omitted rows and
columns.

\subsection{A first-jet Hecke correspondence}
\label{sec:wild-jet-hecke-root}

Put $R=\mathbb F_q[X]/(X^2)$. A first-jet flag in a rank-two
bundle is a free rank-one direct summand of its fiber over $R$.
Thus, it records a flag at $X=0$ together with its first-order
variation. For example, $R(1,uX)\subset R^2$ reduces to the line
$\mathbb F_q(1,0)$ at zero, while $u$ specifies the first-order
variation.
Fix an odd prime power $q$, a nontrivial additive character $\chi$,
$\xi\in\mathbb F_q\setminus\{0,1\}$, and $c\ne0$.
At zero use a first-jet flag, and at $1$ and infinity use ordinary
flags. In degree zero let $E_u=\mathcal O\oplus\mathcal O$, with flags
\[
 (1,uX)\pmod{X^2},\qquad(1,1),\qquad(0,1).
\]
In degree one let $F_U=\mathcal O\oplus\mathcal O(1)$, with flags
\[
 (1,UX)\pmod{X^2},\qquad(1,1),\qquad(1,0),
\]
where the last vector is in the infinity frame. At zero put
\[
 \vartheta(a_0+a_1X+\cdots)=\chi(ca_1/a_0)
\]
on the flag line, and use the trivial character on the quotient
and on the ordinary flags.

\begin{lemma} \label{lem:first-jet-modifications}
For $u,U\notin\{0,1\}$, the elementary modifications
$E_u\hookrightarrow F_U$ at $\xi$ that preserve the specified flags
are, up to a common scalar,
\begin{equation}\label{eq:first-jet-injection}
 G_h(X)=\begin{pmatrix}-h&1\\-hX/z_1&\xi/z_1\end{pmatrix},
 \qquad z_1=\frac{h-\xi}{h-1},
\end{equation}
where $h\in\mathbb F_q\setminus\{0,1,\xi\}$ satisfies
\begin{equation}\label{eq:first-jet-output}
 U h(h-\xi)=(h-1)(h-\xi u).
\end{equation}
Each such $h$ occurs once, and its coefficient on the flag line at
zero is $\chi(-cu/h)$.
Both bundles with level structure have automorphism group $\mathbb F_q^\times$,
and distinct parameters give nonisomorphic objects within each family.
These scalars act trivially on the coefficient lines, and the
coordinate sections have the same Petersson mass.
\end{lemma}

\begin{proof}
For $E_u$, the three reduced flag lines are distinct, so their
common stabilizer is the scalar group. For $F_U$, an automorphism is
$\left(\begin{smallmatrix}a&0\\b+\gamma X&e\end{smallmatrix}\right)$;
the flags at infinity, zero, and one successively force $\gamma=0$, $b=0$,
and $a=e$. The same calculation applies to an isomorphism between
two objects in either family. Since a scalar does not change the
first-jet parameter, the parameters give distinct isomorphism classes.

The character $\vartheta$ defines a character of the compact subgroup
preserving the first-jet flag. It is compatible with a global
central character: with $K=\mathbb F_q(X)$, set
\[
 \omega((a_s)_s)=
 \chi\left(-\sum_s\operatorname{Tr}_{k_s/\mathbb F_q}
      \operatorname{Res}_s\left(\frac cX\frac{da_s}{a_s}\right)\right).
\]
This is trivial on $K^\times$ by the residue theorem, restricts to
$\vartheta^{-1}$ on the scalar units at zero, and is unramified
elsewhere. Thus, the compact type satisfies
$\tau(z)=\omega(z)^{-1}$ on the compact center.
The local type excludes every determinant character: varying the two
diagonal units independently would force the same character to equal
both $\vartheta$ and the trivial character. Constant scalars have
trivial action on the coefficient lines.

The map \eqref{eq:first-jet-injection} acts on the first-jet flag by
\[
 G_h(X)(1,uX)
   =(-h+uX)\left(1,\frac{h-\xi u}{hz_1}X\right)
        \pmod{X^2}.
\]
Thus, the image flag satisfies \eqref{eq:first-jet-output}, and its
coefficient is $\vartheta(-h+uX)=\chi(-cu/h)$.

We verify that the list of modifications is complete. A bundle map
$E_u\to F_U$ has constant first-row entries and affine-linear
second-row entries. Preservation of the infinity flag and of the
reduced zero flag forces the form
\[
 G(X)=\begin{pmatrix}a&b\\\gamma X&e\end{pmatrix}.
\]
The isomorphism at infinity gives $b\ne0$, and the isomorphism at zero gives $a,e\ne0$. Divide by the common scalar to put $b=1$ and
write $a=-h$. A length-one cokernel at $\xi$ gives
$-he=\gamma\xi$, while preservation of the flag at $1$ gives
$\gamma+e=1-h$. Solving yields
\[
\gamma=-\frac{h(h-1)}{h-\xi}=-h/z_1,\qquad
 e=\frac{\xi(h-1)}{h-\xi}=\xi/z_1.
\]
The isomorphism and flag conditions exclude $h=0,1,\xi$.
The first-jet flag condition is \eqref{eq:first-jet-output}.
Conversely, these formulas give a bundle injection with determinant
$(h/z_1)(X-\xi)$, isomorphic at all marked points.
Its kernel at $\xi$ is $\langle(1,h)\rangle$, so distinct $h$ give
distinct elementary modifications. Its image at $\xi$ is
$\langle(1,\xi/z_1)\rangle$. Since $\xi\ne0,1$, these lines are distinct for distinct $h$. These image lines distinguish the right cosets in
$H_\xi$. A repeated root contributes one coset.

For $u,U\notin\{0,1\}$, the polynomial in
\eqref{eq:first-jet-output} is nonzero at $h=0,1,\xi$; its projective
root at infinity can occur only when $U=1$. Thus, no modification
is lost on the stated interior indices.
The equal scalar automorphism groups give the coordinate sections
the same Petersson mass, proportional to $(q-1)^{-1}$. The earlier
central-quotient argument applies, so the scalar action is trivial
and both normalized coordinate embeddings are isometric.

For clarity, choose the following local representatives. Write
$S_u(X)=\left(\begin{smallmatrix}1&0\\uX&1\end{smallmatrix}\right)$
and $g_1=\left(\begin{smallmatrix}1&0\\1&1\end{smallmatrix}\right)$.
Take $g_{E_u,0}=S_u$, $g_{F_U,0}=S_U$, and
$g_{E_u,1}=g_{F_U,1}=g_1$. At infinity take
$g_{E_u,\infty}=w_0$ and
$g_{F_U,\infty}=\operatorname{diag}(1,X)$; elsewhere take the identity.
For a modification $G_h$, let $a_h$ have component $G_h$ at $\xi$
and the identity elsewhere. Then
\[
 g_{F_U}a_h=G_hg_{E_u}k_h,\qquad k_{h,\xi}=I,\qquad
 k_{h,v}^{-1}=g_{F_U,v}^{-1}G_hg_{E_u,v}\quad(v\ne\xi).
\]
The flag conditions put every $k_{h,v}$ in the specified compact
subgroup. Hence, Lemma~\ref{lem:typed-hecke-compression} assigns to this
coset the coefficient $\tau(k_h)^{-1}$, exactly the factor on the flag line
computed above. For a fixed target and $h$, equation
\eqref{eq:first-jet-output} determines $u$ uniquely; thus, each contributing right coset occurs once for each fixed target.
\end{proof}

\subsection{Square parameters in the rank-one matrix}
\label{sec:rankone-square-hecke}

Changing the tame characters in the first-jet model gives the
Sali\'e weight needed for the rank-one matrix at square parameters.

\begin{theorem}\label{thm:rankone-square-hecke}
Let $\operatorname{char}\mathbb F_q>3$, let $c\in\mathbb F_q^{\times2}$, and define
\[
 L_c(H,K)=\sum_{z\ne0}\chi\left(z+
       \frac{c(H^2+HK+K^2-1/3)}z\right).
\]
Then $\|L_c\|\le10q$ in counting measure.
\end{theorem}

\subsubsection*{A weighted first-jet matrix}
Let $\xi\ne0,1$ and $c_0\ne0$. Complete the first-jet correspondence
by the following explicit matrix on all $u,U\in\mathbb F_q$:
\begin{equation}\label{eq:rankone-completed-jet}
 A(u,U)=
 \sum_{\substack{h\in\mathbb F_q\setminus\{0,\xi\}:\\
 U h(h-\xi)=(h-1)(h-\xi u)}}
 \eta\left(\frac h{h-\xi}\right)\chi(-c_0u/h)
       +\1_{U=1}.
\end{equation}
The last term accounts for the projective root $h=\infty$ of the homogenized polynomial. The term $h=1$ at $U=0$ is also included in $A$. These terms need not correspond to injections between our bundle charts. Only the interior block $u,U\notin\{0,1\}$ is identified with a Hecke
compression. We bound the added rows and columns entry by entry.

\begin{lemma}\label{lem:rankone-completed-jet-bound}
The matrix \eqref{eq:rankone-completed-jet} satisfies
$\|A\|\le6\sqrt q$.
\end{lemma}
\begin{proof}
Use the bundles with level structure from Lemma~\ref{lem:first-jet-modifications},
with $c=c_0$. Retain the wild character
$\vartheta(a_0+a_1X)=\chi(c_0a_1/a_0)$ on the flagged line at zero.
At $1$ put $\eta$ on the flagged line, and at infinity put $\eta$
on the quotient by the flag. The automorphic central character is
the product of the residue character associated with $-c_0/X$ and
the quadratic tame-symbol character associated with $X-1$ and
$\eta$. It is therefore inverse to the character of the compact
center on the coefficient line.
This type excludes residual determinant characters at zero.
Proposition~\ref{prop:weighted-hecke-bound} therefore bounds the
unnormalized Hecke operator at $\xi$ by $2\sqrt q$.
The two added tame factors have product $\eta(a)^2=1$ on a constant
scalar $a$. Thus, the scalar stabilizers still act trivially, and
the coordinate embeddings from Lemma~\ref{lem:first-jet-modifications}
remain isometric.

For the injection $G_h$ in \eqref{eq:first-jet-injection}, the ordinary
flag line at $1$ has scaling factor $1-h$. With the chosen infinity
representatives, the framed map has reduction
\[
 \left.g_{F_U,\infty}^{-1}G_hg_{E_u,\infty}\right|_{\infty}
 =\begin{pmatrix}1&-h\\0&-h/z_1\end{pmatrix}.
\]
Its map on the quotient line has scaling factor $-h/z_1$, where
$z_1=(h-\xi)/(h-1)$. Their product has quadratic character
\[
 \eta\bigl((1-h)(-h/z_1)\bigr)
       =\eta\left(\frac h{h-\xi}\right).
\]
The first-jet flag at zero gives $\chi(-c_0u/h)$.
Let $A_{\mathrm{int}}$ be the restriction of
\eqref{eq:rankone-completed-jet} to $u,U\notin\{0,1\}$. By
Lemma~\ref{lem:first-jet-modifications} and
Lemma~\ref{lem:typed-hecke-compression},
\begin{equation}\label{eq:rankone-jet-exact-compression}
 i_F^*H_\xi i_E=A_{\mathrm{int}}^{\top}.
\end{equation}
The transpose occurs because $A(u,U)$ has the source parameter as
its row index, whereas the compression has the target as its row
index. The ordinary transpose preserves the norm of a complex
matrix, since $A^{\top}=\overline{A^*}$. Thus,
$\|A_{\mathrm{int}}\|\le2\sqrt q$.

For arbitrary $u,U$, the homogenized quadratic polynomial is
not identically zero. Indeed, its three coefficients are
$U-1$, $1+\xi u-\xi U$, and $-\xi u$; their simultaneous vanishing
would imply $U=1,u=0,\xi=1$. It has at most two projective roots.
Thus, every entry of $A$ has modulus at most two, including its
added root at infinity. The two omitted rows and columns have Hilbert--Schmidt norm
at most $4\sqrt q$. Adding the two bounds gives
$\|A\|\le6\sqrt q$.
\end{proof}

\subsubsection*{Fourier transform of the first-jet matrix}
Let $\mathcal F_1(U,y)=q^{-1/2}\chi(Uy)$, and set
\[
 a(u,y)=\frac{(\xi-1)y(1-u)}\xi,
 \qquad b(u,y)=\frac{u(c_0+y)}\xi,
 \qquad C(u,y)=\frac{y+uy(\xi-2)-c_0u}\xi.
\]
We have
\begin{equation}\label{eq:rankone-jet-salie}
 (A\mathcal F_1)(u,y)=
 q^{-1/2}\chi(C(u,y))
       \sum_{s\ne0}\eta(s)\chi\bigl(a(u,y)s+b(u,y)/s\bigr).
\end{equation}
To verify it, sum first over $U$ in
\eqref{eq:rankone-completed-jet} and make the M\"obius substitution
$s=h/(h-\xi)$. The allowed projective $h$ values correspond exactly
to $s\in\mathbb F_q^\times$; in particular, $h=\infty$ becomes $s=1$.
Before the substitution the phase is
\[
 y-\frac{u(c_0+y)}h+
       \frac{(\xi-1)y(1-u)}{h-\xi}.
\]
The resulting phase is $C+as+b/s$ and the character is
$\eta(s)$.

Writing $G_\eta=\sum_z\chi(z^2)$, the Sali\'e identity gives, when
$b(u,y)\ne0$,
\begin{equation}\label{eq:rankone-jet-root-product}
 (A\mathcal F_1)(u,y)=\frac{G_\eta}{\sqrt q}\eta(b(u,y))\chi(C(u,y))
                \sum_{v^2=4a(u,y)b(u,y)}\chi(v).
\end{equation}
The factor $G_\eta/\sqrt q$ is unimodular.

\begin{proof}[Proof of Theorem~\ref{thm:rankone-square-hecke}]
Choose $k\in\mathbb F_q^\times$ with $k^2=c$. In the preceding
identities put
\[
 \xi=4,\qquad c_0=16k/3,\qquad
 u=\frac12+\frac{3H}4,\qquad y=-\frac{8k}3-\frac{4Y}3.
\]
These are bijective affine changes of the row and column variables.
Direct calculation gives
\[
 4a(u,y)b(u,y)=\frac{(9H^2-4)(4c-Y^2)}{12},
\]
and
\[
 C(u,y)=-HY/2-2kH-2Y/3-2k.
\]
On the other hand, the exact Fourier calculation in
Subsection~\ref{sec:rankone-affine-fourier} gives
\[
 (L_c\mathcal F_1)(H,Y)=
 \sqrt q\,\eta\left(-\frac{9H^2-4}{12}\right)\chi(-HY/2)
       \sum_{v^2=(9H^2-4)(4c-Y^2)/12}\chi(v),
\]
provided $9H^2-4\ne0$.
Remove $H=\pm2/3$ and $Y=\pm2k$. On the remaining indices,
$L_c\mathcal F_1$ is exactly $\sqrt q$ times
$A\mathcal F_1$, after the displayed affine coordinate permutations
and row and column factors of modulus one. Indeed, the additional phase
$-2kH-2Y/3-2k$ separates, and the quotient of the quadratic-character
factors is
\[
 \eta\left(\frac{1-u}{3(c_0+y)}\right),
\]
which also separates into row and column factors of modulus one.
The remaining constant $(G_\eta/\sqrt q)^{-1}$ is also unimodular.
Lemma~\ref{lem:rankone-completed-jet-bound} therefore bounds this
interior matrix by $6q$.

The displayed Fourier formula shows that every entry of
$L_c\mathcal F_1$ has modulus at most $2\sqrt q$, including the
exceptional values: when its Sali\'e coefficient $B_H$ vanishes,
the remaining sum is a quadratic Gauss sum or zero.
The two omitted rows and columns have Hilbert--Schmidt norm
at most $4q$.
Since $\mathcal F_1$ is unitary, $\|L_c\|\le10q$ follows.
\end{proof}

\subsection{Nonsquare parameters in the rank-one matrix}
\label{sec:rankone-nonsplit-type}

For nonsquare parameters, we replace the split first-jet character
by a nonsplit semisimple local type. The Hecke point is still
rational. This gives the estimate for all nonzero parameters.

\begin{theorem}\label{thm:rankone-all-parameters}
For $\operatorname{char}\mathbb F_q>3$ and every $c\in\mathbb F_q^\times$, the matrix
\[
 L_c(H,K)=\sum_{z\ne0}\chi\left(z+
       \frac{c(H^2+HK+K^2-1/3)}z\right)
\]
satisfies $\|L_c\|\le10q$. If $c$ is nonsquare, the proof below gives
$\|L_c\|\le(2+\sqrt2)q$.
\end{theorem}

\begin{proof}
The square case is Theorem~\ref{thm:rankone-square-hecke}. We prove
the nonsquare case by a nonsplit local type.

\subsubsection*{The local character and its automorphic Hilbert space}
Let $C\in\mathbb F_q^\times$ be nonsquare and put
\[
 J_C=\begin{pmatrix}0&C\\1&0\end{pmatrix},\qquad
 \mathcal T=\{aI+bJ_C:(a,b)\in\mathbb F_q^2\setminus\{(0,0)\}\}.
\]
This is the nonsplit torus centralizing $J_C$; every displayed nonzero
matrix is invertible. At the place $X=0$, take the compact subgroup
whose reduction modulo $X$ lies in $\mathcal T$, and work modulo
$X^2$. Every element has a unique form
$\mathsf B(I+XY)$ modulo $X^2$, with $\mathsf B\in\mathcal T$ and
$Y\in\operatorname{Mat}_2(\mathbb F_q)$. Define
\begin{equation}\label{eq:rankone-nonsplit-local-character}
 \tau_{\mathrm{ns}}\bigl(\mathsf B(I+XY)\bigr)
       =\eta(\det \mathsf B)\chi(\operatorname{Tr}(J_C Y)).
\end{equation}
This is a character: conjugation by $\mathsf B$ preserves $J_C$ and hence the
trace pairing. Its restriction to scalar units is trivial, since
$\eta(a^2)=1$ and $\operatorname{Tr}J_C=0$.

At $X=1$ use an ordinary flag with character $\eta$ on its line;
at infinity use an ordinary flag with character $\eta$ on its
quotient. Use maximal compact level elsewhere. Let $\omega$ be the
quadratic tame-symbol character associated with $X-1$. It is
unramified at zero and is its own inverse. Consequently,
$\tau(z)=\omega(z)^{-1}$ on the compact center, so these types
define an automorphic $L^2$ space with the fixed convention above.

No determinant character occurs in this type, irrespective of its
conductor. Indeed, the local element $I+XsE_{12}$ has determinant
exactly one, whereas \eqref{eq:rankone-nonsplit-local-character}
takes the value $\chi(s)$ on it, since
$\operatorname{Tr}(J_C E_{12})=1$. Choosing $s$ with $\chi(s)\ne1$
excludes every representation of the form $\mu\circ\det$.
Proposition~\ref{prop:weighted-hecke-bound} therefore gives
$\|H_\xi\|\le2\sqrt q$ at every rational
$\xi\notin\{0,1\}$.

\subsubsection*{The matrix of the Hecke compression}
In degree zero take $E_H=\mathcal O\oplus\mathcal O$, with the
standard local frame at zero with the character
\eqref{eq:rankone-nonsplit-local-character}, and flags
\[
 (1,H)\text{ at }1,\qquad e_2\text{ at infinity}.
\]
In degree one take $F_K=\mathcal O\oplus\mathcal O(1)$, with the
same standard local type at zero, and flags
\[
 (1,K)\text{ at }1,\qquad e_1\text{ at infinity}.
\]
Here $H,K\in\mathbb F_q$. In the target frame at infinity, the second affine coordinate is divided by $X$.

Both automorphism groups are precisely $\mathbb F_q^\times$.
For $E_H$, an automorphism preserving the local type has reduction
in $\mathcal T$; fixing $e_2$ forces it to be scalar.
For $F_K$, an automorphism has the form
$\left(\begin{smallmatrix}a&0\\b+\gamma X&e\end{smallmatrix}\right)$.
Its infinity flag forces $\gamma =0$, and its constant reduction lies in
$\mathcal T$ only when $b=0$ and $a=e$.
The scalar character is trivial, so the coefficient lines are
well defined on both families. The same calculations show that
different finite parameters give nonisomorphic objects.
Their coefficient lines have equal Petersson mass, a common
multiple of $1/(q-1)$ by the central-quotient argument above. Let
$i_E,i_F$ be the corresponding isometric embeddings from counting
measure.

For $t\in\mathbb F_q$ put $D_t=t^2-C\ne0$ and define
\begin{equation}\label{eq:rankone-nonsplit-injection}
 G_t(X)=\begin{pmatrix}
 t&C\\1+D_tX/(C\xi)&t
 \end{pmatrix}.
\end{equation}
It has
\[
 \det G_t(X)=D_t(1-X/\xi),\qquad
 \ker G_t(\xi)=\langle(1,-t/C)\rangle.
\]
Thus, it is an elementary upper modification at $\xi$.
Its value at zero is $tI+J_C\in\mathcal T$, so it identifies the
specified local coefficient fibers. Its infinity flag is $e_1$.
Its image flag at $1$ has finite coordinate
\begin{equation}\label{eq:rankone-nonsplit-output}
 K=\frac{1+D_t/(C\xi)+tH}{t+CH}.
\end{equation}
When $t+CH=0$, the numerator is
$(\xi-1)(1-CH^2)/\xi\ne0$, so the image flag is the omitted
point at infinity.

To see that this list is exhaustive, start with a general bundle map
\[
 G(X)=\begin{pmatrix}a&b\\c+\gamma X&e+fX\end{pmatrix}.
\]
Preservation of the infinity flags forces $f=0$ and $b\ne0$.
Its reduction at zero must belong to $\mathcal T$, so $e=a$ and
$b=Cc$, with $c\ne0$. Dividing by the scalar $c$, which has trivial
total coefficient character, gives
$\left(\begin{smallmatrix}t&C\\1+\gamma X&t\end{smallmatrix}\right)$.
Its determinant has its unique simple zero at $\xi$ precisely when
$\gamma=(t^2-C)/(C\xi)$, giving \eqref{eq:rankone-nonsplit-injection}.
Thus, these formulas account for every elementary modification that
contributes to the stated source and target families.
The $q$ kernel lines $\langle(1,-t/C)\rangle$ are distinct and
exhaust all kernel lines except $e_2$. The image lines at $\xi$ are
$\langle(C,t)\rangle$, so the corresponding right cosets are also
distinct. The omitted kernel line $e_2$ gives a modification whose
infinity flag lies in the degree-one summand, so it cannot belong to
the target family, whose infinity flag is transverse to that
summand. For each of the remaining lines,
\eqref{eq:rankone-nonsplit-injection} gives the
unique elementary modification. The finite coordinates of its image flag are given by \eqref{eq:rankone-nonsplit-output}; the equal scalar
stabilizers introduce no multiplicity or measure factor.

Take identity representatives at zero,
$g_{E_H,1}=\left(\begin{smallmatrix}1&0\\H&1\end{smallmatrix}\right)$,
and $g_{F_K,1}=\left(\begin{smallmatrix}1&0\\K&1\end{smallmatrix}\right)$.
At infinity take $g_{E_H,\infty}=w_0$ and
$g_{F_K,\infty}=\operatorname{diag}(1,X)$, and use the identity
elsewhere. If $a_t$ is supported at $\xi$ with component $G_t$, then
\[
 g_{F_K}a_t=G_tg_{E_H}k_t,\qquad k_{t,\xi}=I,\qquad
 k_{t,v}^{-1}=g_{F_K,v}^{-1}G_tg_{E_H,v}\quad(v\ne\xi).
\]
These elements belong to the specified compact subgroups. For fixed $t$ and a finite image coordinate
$K$, the isomorphism $G_t(1)$ determines at most one source flag $H$.
Thus, each right coset contributes once to the full coordinate matrix,
and its coefficient is $\tau(k_t)^{-1}$.

Write $G_t(X)=G_t(0)(I+XY_t)$. Direct multiplication gives
\[
 Y_t=\begin{pmatrix}-1/\xi&0\\t/(C\xi)&0\end{pmatrix},
 \qquad\operatorname{Tr}(J_C Y_t)=t/\xi.
\]
The local coefficient at zero is consequently
$\eta(D_t)\chi(t/\xi)$. The line at $1$ scales by $t+CH$. At
infinity the framed map has reduction
\[
 \left.g_{F_K,\infty}^{-1}G_tg_{E_H,\infty}\right|_{\infty}
 =\begin{pmatrix}C&t\\0&D_t/(C\xi)\end{pmatrix},
\]
so the quotient scales by $D_t/(C\xi)$. Their product is
\[
 \eta(D_t)\eta(t+CH)\eta(D_t/(C\xi))\chi(t/\xi)
 =\eta(C\xi)\eta(t+CH)\chi(t/\xi).
\]
Define
\[
 \mathsf R(H,K)=
 \sum_{t^2+C\xi(H-K)t-C^2\xi HK+C(\xi-1)=0}
      \eta(t+CH)\chi(t/\xi).
\]
Lemma~\ref{lem:typed-hecke-compression} gives the exact identity
\begin{equation}\label{eq:rankone-nonsplit-exact-compression}
 i_F^*H_\xi i_E=\eta(C\xi)\mathsf R^{\top}.
\end{equation}
Here again the displayed matrix uses the source as its row index;
its ordinary transpose has the same norm. Each root is counted once.
Consequently,
\begin{equation}\label{eq:rankone-nonsplit-root-bound}
 \|\mathsf R\|\le2\sqrt q.
\end{equation}

\subsubsection*{Fourier transformation and the affine Kloosterman matrix}
Set $z=t+CH$. Equation \eqref{eq:rankone-nonsplit-output} becomes
\[
 K=\frac z{C\xi}+\frac{(\xi-2)H}{\xi}
                  +\frac{(\xi-1)(1-CH^2)}{\xi z}.
\]
The image flag has a finite coordinate exactly when $z\in\mathbb F_q^\times$.
Consequently, with the same unitary Fourier matrix as before,
\[
 (\mathsf R\mathcal F_1)(H,Y)
  =q^{-1/2}\chi\left(\frac{(\xi-2)HY-CH}{\xi}\right)
       \sum_{z\ne0}\eta(z)\chi(az+b/z),
\]
where
\[
 a=\frac{C+Y}{C\xi},\qquad
 b=\frac{Y(\xi-1)(1-CH^2)}{\xi}.
\]
When $Y\ne0$, the Sali\'e identity applies because
$1-CH^2\ne0$.

Let the Kloosterman parameter $c$ be nonsquare and put
\[
 C=64c/9,\qquad\xi=4,\qquad
 H=-\frac{3Y_0}{16c},\qquad
 Y=-\frac{32c}{9}+\frac{16c}{3}H_0.
\]
Then
\[
 4ab=\frac{(9H_0^2-4)(4c-Y_0^2)}{12},
 \qquad
 \frac{(\xi-2)HY-CH}{\xi}=-H_0Y_0/2+2Y_0/3.
\]
The row and column roles are exchanged by these substitutions.
The resulting expression contains the same sum over roots as $L_c\mathcal F_1$ in
Subsection~\ref{sec:rankone-affine-fourier}. By the Sali\'e identity, the factor outside the root sum for $\mathsf R\mathcal F_1$ has modulus one. The corresponding factor for $L_c\mathcal F_1$ has modulus $\sqrt q$.
After removing $H_0=\pm2/3$, the ratio of its character factors is
\[
 \eta\left(-\frac{3H_0+2}{4(4c-Y_0^2)}\right),
\]
which separates into row and column factors of modulus one.
The additional phase $2Y_0/3$ depends only on the column, and $G_\eta/\sqrt q$
is unimodular. Since $c$ is nonsquare, $4c-Y_0^2$ never vanishes.
Thus, \eqref{eq:rankone-nonsplit-root-bound} bounds this part of
$L_c\mathcal F_1$ by $2q$, after transposition, coordinate permutations,
and multiplication by unitary diagonal matrices.

On its two omitted rows, the original Sali\'e expression reduces
to a Gauss sum and every entry has modulus at most $\sqrt q$.
Their Hilbert--Schmidt norm is at most $\sqrt2\,q$. This proves
$\|L_c\|\le(2+\sqrt2)q$ for nonsquare $c$.
The square case is Theorem~\ref{thm:rankone-square-hecke}; together
they prove Theorem~\ref{thm:rankone-all-parameters}.
\end{proof}

\subsection{The rank-one matrix in characteristic three}
\label{sec:rankone-characteristic-three}

In characteristic three, the preceding change of variables involving $1/3$ is not defined.
We instead transform the unshifted matrix in both variables and
use multiplicative characters of a split or a nonsplit torus.
We use the following consequence of the
mixed character-sum estimate of Fu and Wan
\cite[Theorem~4.6]{FuWanIncomplete}.

\Needspace{8\baselineskip}
\begin{lemma}
\label{lem:rankone-mixed-character-sums}
Let $\chi$ be a nontrivial additive character of an odd finite field
$\mathbb F_q$.
For multiplicative characters $\alpha,\beta$ of $\mathbb F_q^\times$
and $\gamma\ne0$,
\begin{equation}\label{eq:rankone-split-mixed-sum}
 \left|\sum_{s\ne0,-1}\alpha(s)\beta(s+1)\chi(\gamma s)\right|
 \le2\sqrt q.
\end{equation}
If $\delta\in\mathbb F_{q^2}\setminus\mathbb F_q$, $\Omega$ is any
multiplicative character of $\mathbb F_{q^2}^\times$, and
$\gamma\ne0$, then
\begin{equation}\label{eq:rankone-nonsplit-mixed-sum}
 \left|\sum_{s\in\mathbb F_q}\Omega(s+\delta)\chi(\gamma s)\right|
 \le2\sqrt q.
\end{equation}
\end{lemma}

\begin{proof}
For \eqref{eq:rankone-split-mixed-sum}, choose a multiplicative
character $\rho$ of order $q-1$ and integers $1\le a,b\le q-1$
with $\alpha=\rho^a$ and $\beta=\rho^b$. Apply the cited theorem
over the base field to $f(s)=s^a(s+1)^b$ and $g(s)=\gamma s$.
The sum of the degrees of the distinct irreducible factors of $f$
is two, and $g$ is a linear polynomial. The theorem gives
$2\sqrt q$, including when either multiplicative character is
trivial.
For \eqref{eq:rankone-nonsplit-mixed-sum}, apply the same theorem
with extension degree two, $f(s)=s+\delta$, and $g(s)=\gamma s/2$.
Then $\operatorname{Tr}_{\mathbb F_{q^2}/\mathbb F_q}g(s)=\gamma s$,
and its bound is again $2\sqrt q$.
In both applications the required nondegeneracy condition holds:
a nonconstant linear polynomial is not of the form $r^p-r$ over
an algebraic closure, because every pole of such a nonconstant
rational function has order divisible by the characteristic $p$.
\end{proof}

\begin{theorem}\label{thm:rankone-characteristic-three}
Suppose $\operatorname{char}\mathbb F_q=3$. For $c\ne0$, the matrix
\[
 L_c^{(3)}(h,k)=\sum_{z\ne0}\chi\left(z+
             \frac{c((h-k)^2+h+k)}z\right),
 \qquad h,k\in\mathbb F_q,
\]
satisfies $\|L_c^{(3)}\|\le4q$ in counting measure.
If $c$ is nonsquare, $\|L_c^{(3)}\|\le2q$.
\end{theorem}

\begin{proof}
Let $\mathcal F_1(h,a)=q^{-1/2}\chi(ha)$, and define
\[
 \mathsf S_c(a,b)=
 \begin{cases}
 \displaystyle\eta(a-b)\chi\left(\frac{ab-c}{a-b}\right),&a\ne b,\\
 0,&a=b.
 \end{cases}
\]
Put $G_\eta=\sum_u\chi(u^2)$, so $|G_\eta|=\sqrt q$.
We first claim that
\begin{equation}\label{eq:rankone-characteristic-three-fourier}
 (\mathcal F_1^*L_c^{(3)}\mathcal F_1)(a,b)
   =G_\eta\eta(-1)\chi(a-b)\mathsf S_c(a,b).
\end{equation}
Indeed, in the Fourier sum write $h=u+v$, $k=v$.
Since $2=-1$, summing over $v$ forces
$b-a=c/z$. There is no such $z$ when $a=b$.
For $a\ne b$, the remaining sum is
\[
 \sum_u\chi\left(\frac c{b-a}+(b-a)u^2+(a+b)u\right)
 =G_\eta\eta(b-a)
   \chi\left(\frac{c-(a+b)^2}{b-a}\right).
\]
Here $4=1$, so completing the square gives the displayed expression.
Finally,
\[
 \frac{c-(a+b)^2}{b-a}
 =a-b+\frac{ab-c}{a-b}
\]
in characteristic three, proving
\eqref{eq:rankone-characteristic-three-fourier}.
The factor $\chi(a-b)$ separates into unitary row and column
factors. Consequently,
\begin{equation}\label{eq:rankone-characteristic-three-norm-transfer}
 \|L_c^{(3)}\|=\sqrt q\,\|\mathsf S_c\|.
\end{equation}

\subsubsection*{Square parameters}
Write $c=k^2$, with $k\ne0$, and first remove the rows and columns
$a,b=\pm k$. The changes of variables
\[
 a=k\frac{u+1}{u-1},\qquad b=k\frac{v+1}{v-1}
\]
identify the remaining indices with $u,v\in\mathbb F_q^\times
\setminus\{1\}$. They give
\[
 a-b=\frac{2k(v-u)}{(u-1)(v-1)},\qquad
 \frac{ab-k^2}{a-b}=k\frac{u+v}{v-u}.
\]
After multiplication by row and column factors of modulus one, the resulting matrix is a
compression of the matrix on $\mathbb F_q^\times$ with entries
\[
 \eta(v-u)\chi\left(k\frac{u+v}{v-u}\right)
 =\eta(u)\mathcal K(v/u),
\]
where
\[
 \mathcal K(r)=
 \begin{cases}
 \displaystyle\eta(r-1)\chi\left(k\frac{1+r}{r-1}\right),&r\ne1,\\
 0,&r=1.
 \end{cases}
\]
Multiplication by $\eta(u)$ is unitary. The matrix with entries
$\mathcal K(v/u)$ is diagonalized by the multiplicative characters
$\nu$ of $\mathbb F_q^\times$, with eigenvalues
\[
 \sum_{r\ne0,1}\mathcal K(r)\nu(r)
 =\chi(k)\sum_{s\ne0,-1}
       (\eta\overline\nu)(s)\nu(s+1)\chi(2ks),
 \qquad s=\frac1{r-1}.
\]
Since $2k\ne0$, Lemma~\ref{lem:rankone-mixed-character-sums}
bounds their moduli by $2\sqrt q$. Thus, the retained block of
$\mathsf S_c$ has norm at most $2\sqrt q$.
Every entry of $\mathsf S_c$ has modulus at most one. The union of
the two removed rows and columns contains at most $4q$ entries,
so its Hilbert--Schmidt norm is at most $2\sqrt q$.
We conclude that $\|\mathsf S_c\|\le4\sqrt q$.

\subsubsection*{Nonsquare parameters}
Let $\mathbb E=\mathbb F_{q^2}=\mathbb F_q(\delta)$, with $\delta^2=c$,
and put $\mathcal T=\mathbb E^\times/\mathbb F_q^\times$.
Choose a multiplicative character $\lambda$ of $\mathbb E^\times$
whose restriction to $\mathbb F_q^\times$ is $\eta$; such an
extension exists because these groups are cyclic.
The elements $[s+\delta]$, $s\in\mathbb F_q$, are exactly the
nonidentity elements of $\mathcal T$, each written once.
Define a function on this group by
\[
 \mathcal K(1)=0,\qquad
 \mathcal K([w+\delta])=\lambda(w+\delta)^{-1}\chi(-w).
\]
For $s\ne t$, put $w=(st-c)/(t-s)$. The identity
\[
 \frac{s+\delta}{t+\delta}=\frac{t-s}{t^2-c}(w+\delta)
\]
gives
\begin{equation}\label{eq:rankone-nonsplit-torus-compression}
 \mathsf S_c(s,t)
 =\eta(-1)\lambda(s+\delta)\lambda(t+\delta)^{-1}\eta(t^2-c)
      \mathcal K([s+\delta][t+\delta]^{-1}).
\end{equation}
Both sides vanish when $s=t$. All factors outside $\mathcal K$
in \eqref{eq:rankone-nonsplit-torus-compression} have modulus one and
separate into row and column factors. Hence, $\mathsf S_c$ has the
same norm as a compression of convolution by $\mathcal K$ on
$\mathcal T$.

The characters $\nu$ of $\mathcal T$ diagonalize this convolution.
Its eigenvalues, after replacing $\nu$ by $\nu^{-1}$ if necessary, are
\[
 \sum_{w\in\mathbb F_q}\lambda(w+\delta)^{-1}\nu([w+\delta])\chi(-w)
 =\sum_{w\in\mathbb F_q}\Omega(w+\delta)\chi(-w),
\]
where $\Omega(z)=\lambda(z)^{-1}\nu([z])$ is a multiplicative
character of $\mathbb E^\times$.
Lemma~\ref{lem:rankone-mixed-character-sums} bounds each eigenvalue
by $2\sqrt q$. Therefore, $\|\mathsf S_c\|\le2\sqrt q$.
The asserted estimates now follow from
\eqref{eq:rankone-characteristic-three-norm-transfer}.
\end{proof}

\subsection{The shifted-shell estimate in even dimensions}
\label{sec:completed-generic-blocks}

The arithmetic estimates above concern matrices indexed by at most two finite-field variables.
They bound the shifted-shell matrices in every even dimension through
Fourier transformation and induction. Both Witt types are included.

\begin{theorem}\label{thm:completed-generic-blocks}
Let $q$ be an odd prime power, let $d=2m\ge2$, and let $Q$ be a
nondegenerate quadratic form on $\mathbb F_q^d$. For $j\ne0$, put
$S_j=\{x:Q(x)=j\}$ and $\rho_d=|S_j|/q^d$. The matrix
\[
 M_t(x,y)=\1_{S_j}(x+y-t)-\rho_d,\qquad x,y\in S_j,
\]
satisfies
\[
 \|M_t\|_{2\to2}\le C_mq^{m-1}\qquad\text{if }Q(t)\ne j.
\]
The constant depends only on $m$, uniformly in the field, the form,
the nonzero radius, and the translation. One may take $C_1=4$ and,
for $m\ge2$, $C_m=3m+6$.
\end{theorem}

\begin{proof}
All norms use counting measure. Write $Q(x)=x^{\top}Ax$ and let
$\epsilon_Q=\eta((-1)^m\det A)$. Quadratic Gauss sums give
\begin{equation}\label{eq:shifted-shell-density}
 |S_j|=q^{2m-1}-\epsilon_Q q^{m-1},\qquad
 \rho_d=q^{-1}-\epsilon_Q q^{-m-1}.
\end{equation}
For $d=2$, fixing $x\in S_j$ leaves a nonconstant affine linear
equation for $y\in S_j$, since $x-t\ne0$. A nonzero level of a
nondegenerate binary form contains no affine line, so each row and
column of the uncentered matrix has at most two ones. Its norm is
at most two, and
\[
 \|M_t\|\le2+\frac{|S_j|^2}{q^2}<4.
\]
We now assume $m\ge2$ and use induction on $m$.

Choose a hyperbolic plane and write
\begin{equation}\label{eq:shifted-witt-coordinates}
 Q(b,c,z)=bc+P(z),\qquad z\in W,\qquad \dim W=n=d-2.
\end{equation}
Here $P$ is nondegenerate and has the same sign $\epsilon_Q$ in its
even-dimensional counting formulas. If $Q(t)\ne0$, the odd-dimensional
nondegenerate space $t^\perp$, of dimension at least three, contains
a hyperbolic plane. Choose this plane in \eqref{eq:shifted-witt-coordinates}.
Then $t=(0,0,T)$ and $P(T)=Q(t)$. If $t\ne0$ is isotropic, choose a
hyperbolic plane through $t$ and arrange $t=(1,0,0)$. For $t=0$,
choose any hyperbolic plane. Thus, in all cases,
\[
 t=(\tau,0,T),\qquad
 \tau=0\ \text{or}\ (\tau,T)=(1,0),\qquad P(T)\ne j.
\]

Separate $S_j$ into the boundary $c=0$ and the chart $c\ne0$.
The boundary consists of $(b,0,z)$ with $P(z)=j$. The chart has
the bijective parametrization
\begin{equation}\label{eq:shifted-shell-chart}
 (h,u)\longmapsto
 \left(jh-\frac{P(u)}h,\frac1h,\frac uh\right),
 \qquad h\in\mathbb F_q^\times,\quad u\in W.
\end{equation}

\subsubsection*{The boundary and the cross-blocks}
Let $S_j(P)=\{z\in W:P(z)=j\}$. The boundary diagonal block is
\[
 J_q\otimes M_T^P+(\rho_{d-2}-\rho_d)J,
 \qquad \rho_{d-2}=q^{-1}-\epsilon_Q q^{-m},
\]
where $M_T^P$ is the lower-dimensional centered matrix. Since
$|S_j(P)|\le2q^{2m-3}$, induction gives
\begin{equation}\label{eq:shifted-boundary-bound}
 \|M_t|_{c=c'=0}\|
 \le C_{m-1}q^{m-1}+2q^{m-2}.
\end{equation}

Let $\mathcal B$ be the uncentered block from boundary rows
$(b,0,z)$ to chart columns $(h,u)$. Its incidence equation is
\begin{equation}\label{eq:shifted-cross-incidence}
 b-\tau+hP(z-T)+2B(z-T,u)=0.
\end{equation}
Extend the columns by allowing $h=0$, and denote the resulting matrix by $\widetilde{\mathcal B}$. For distinct $z,z'$ with
$P(z)=P(z')=j$, the coefficient vectors
\[
 \bigl(P(z-T),2B(z-T,\cdot)\bigr),\qquad
 \bigl(P(z'-T),2B(z'-T,\cdot)\bigr)
\]
are nonzero and nonparallel. Indeed, proportionality would give
$z'-T=\kappa(z-T)$ with $\kappa\ne0$ and
$\kappa(\kappa-1)P(z-T)=0$. If $\kappa\ne1$, then $P(z-T)=0$;
the identities $P(z)=P(z')=j$ consequently imply
$(\kappa-1)(j-P(T))=0$, a contradiction.

Every row of the extended matrix has $q^n$ ones. Rows with the
same $z$ and different $b$ have disjoint supports, whereas rows
with different $z$ have $q^{n-1}$ common columns. Subtracting the
constant matrix therefore gives the exact Gram identity
\[
 (\widetilde{\mathcal B}-q^{-1}J)(\widetilde{\mathcal B}-q^{-1}J)^*
 =\bigoplus_{P(z)=j}(q^nI_q-q^{n-1}J_q).
\]
Restriction to $h\ne0$ cannot increase the norm, so
$\|\mathcal B-q^{-1}J\|\le q^{n/2}=q^{m-1}$. Changing the
centering from $q^{-1}$ to $\rho_d$ adds a matrix of norm at most
\[
 q^{-m-1}\sqrt{q |S_j(P)|(q-1)q^n}
 \le\sqrt2\,q^{m-5/2}.
\]
In particular, the centered cross-block has norm at most $2q^{m-1}$.
It remains to estimate the diagonal block on the chart.

\subsubsection*{Nonisotropic translations}
Suppose $t=(0,0,T)$ and put $a_0=P(T)\ne0,j$. Substitution in
\eqref{eq:shifted-shell-chart} gives the incidence condition
\[
 P(u-v)=j(h^2+k^2)+(j+a_0)hk
          -2kB(u,T)-2hB(v,T).
\]
Write $u=\alpha T+w$ and $v=\beta T+w'$, where
$w,w'\in W_0=T^\perp\cap W$. The restriction $P_\perp$ to $W_0$
is nondegenerate and has odd dimension $k_0=d-3$. The equation is
\[
 P_\perp(w-w')=H(X,Y),\qquad X=(\alpha,h),\quad Y=(\beta,k),
\]
where
\begin{equation}\label{eq:shifted-binary-phase}
 \begin{aligned}
 H(X,Y)&=X^{\top}DX+Y^{\top}DY-2X^{\top}RY,\\
 D&=\begin{pmatrix}-a_0&0\\0&j\end{pmatrix},\qquad
 R=\begin{pmatrix}-a_0&a_0\\a_0&-(j+a_0)/2\end{pmatrix}.
 \end{aligned}
\end{equation}

In coordinates on $W_0$, write $P_\perp(w)=w^{\top}A_\perp w$
and $P_\perp^*(\xi)=\xi^{\top}A_\perp^{-1}\xi$. The unitary
Fourier transform in $w$ gives the matrix multipliers
\begin{equation}\label{eq:shifted-nonisotropic-multiplier}
(q^{k_0-1}-\rho_dq^{k_0})\1_{\{\xi=0\}}J+\frac{G_\eta^{k_0}\eta(\det A_\perp)}q
  \sum_{s\ne0}\eta(s)
   \chi\left(-\frac{P_\perp^*(\xi)}{4s}\right)
   [\chi(-sH(X,Y))]_{X,Y}.
\end{equation}
These are first considered on all $X,Y\in\mathbb F_q^2$ and then
compressed to $h,k\ne0$. There is no extra normalization factor:
Fourier transformation diagonalizes convolution in $w$ by its
unnormalized Fourier multipliers.

We check the arithmetic reduction, including its scalar. Put
$a=R^{-1}D$ and $\mathcal K=\mathbb F_q[a]$. Direct computation gives
\[
 a^2+a+b=0,\qquad b=\frac{2j}{a_0-j},\qquad
 \det(a^2-I)=b(b+2)\ne0.
\]
The matrix $a$ is nonscalar and invertible, and $1,a^2$ are linearly
independent. Hence, $D,R$ satisfy the hypotheses of
Corollary~\ref{cor:full-root-operator-bound}. Replacing $s$ by $-s$
in the sum in \eqref{eq:shifted-nonisotropic-multiplier} changes it,
up to the unit scalar $\eta(-1)$, into the sum in
\eqref{eq:binary-gaussian-operator-bound} with
$c=P_\perp^*(\xi)/4$. This includes $\xi=0$ and the other frequencies
where $P_\perp^*(\xi)=0$. Since $|G_\eta|=\sqrt q$, the oscillatory
term in \eqref{eq:shifted-nonisotropic-multiplier} has norm at most
\[
 q^{k_0/2-1}\,2q^{3/2}=2q^{m-1}.
\]
The remaining constant matrix, present only at $\xi=0$, has norm
$q^{m-2}$ on all $q^2$ indices. Compression to $h,k\ne0$
decreases the norm. The chart block is consequently bounded by
$3q^{m-1}$.

\subsubsection*{Nonzero isotropic translations}
Suppose $t=(1,0,0)$. The incidence equation is
\[
 P(u-v)=D_j(h,k),\qquad
 D_j(h,k)=j(h^2+hk+k^2)-(h+k).
\]
Write $P(u)=u^{\top}A_Pu$ and $P^*(\xi)=\xi^{\top}A_P^{-1}\xi$.
Because $n=2m-2$ is even, the Fourier multipliers in $u$ are
\begin{equation}\label{eq:shifted-isotropic-multiplier}
 (q^{n-1}-\rho_dq^n)\1_{\{\xi=0\}}J
 +\frac{G_\eta^n\eta(\det A_P)}q
 \left[\sum_{s\ne0}
 \chi\left(-sD_j(h,k)-\frac{P^*(\xi)}{4s}\right)\right]_{h,k}.
\end{equation}
The scalar outside the sum has modulus $q^{m-2}$. If $P^*(\xi)=0$,
the sum is the matrix $q\1_{\{D_j=0\}}-J$. The quadratic equation
has at most two solutions in each row and column, so this matrix
has norm at most $3q$ on all $h,k\in\mathbb F_q$.

If $P^*(\xi)\ne0$, substitute $z=-P^*(\xi)/(4s)$. For
$\operatorname{char}\mathbb F_q\ne3$, put
$H=jh-1/3$ and $K=jk-1/3$. The identity
\[
 jD_j(h,k)=H^2+HK+K^2-1/3
\]
identifies the resulting matrix with that of
Theorem~\ref{thm:rankone-all-parameters}, with
$c=P^*(\xi)/(4j)\ne0$. Its norm is at most $10q$.
In characteristic three, instead put $H=-jh$ and $K=-jk$. Then
\[
 jD_j(h,k)=(H-K)^2+H+K,
\]
so Theorem~\ref{thm:rankone-characteristic-three} bounds the
resulting matrix by $4q$. These affine substitutions are bijections
on the full field; the original nonzero $h,k$ ranges give compressions.

The constant term in \eqref{eq:shifted-isotropic-multiplier} has
norm $q^{m-2}$. It occurs only at $\xi=0$, where the other term
is bounded by $3q^{m-1}$. Thus, the chart block has norm at most
$10q^{m-1}$ in every odd characteristic, including all powers of three.

\subsubsection*{The translation $t=0$}
For $t=0$ the calculation \eqref{eq:shifted-isotropic-multiplier}
holds with $D_j(h,k)=j(h^2+hk+k^2)$. The modes with $P^*(\xi)=0$
are again bounded by $3q^{m-1}$. At the other modes the remaining
matrix, before multiplication by the scalar of modulus $q^{m-2}$,
is
\[
 K_c(h,k)=\sum_{s\ne0}\chi\left(s+
                     \frac{c(h^2+hk+k^2)}s\right),
 \qquad c=\frac{jP^*(\xi)}4\ne0.
\]
If the characteristic is not three, let
$\mathcal F_1(\alpha,h)=q^{-1/2}\chi(\alpha h)$. A binary Gauss
sum gives the exact identity
\begin{equation}\label{eq:shifted-central-double-fourier}
 (\mathcal F_1K_c\mathcal F_1^{\top})(\alpha,\beta)
 =\eta(-3)\bigl(q\1_{\{\alpha^2-\alpha\beta+\beta^2=3c\}}-1\bigr).
\end{equation}
Indeed, the matrix of $h^2+hk+k^2$ has determinant $3/4$ and
inverse $\frac43\bigl(\begin{smallmatrix}1&-1/2\\-1/2&1\end{smallmatrix}\bigr)$.
The two Fourier factors are unitary. The conic matrix on the
right has at most two ones in each row and column, proving
$\|K_c\|\le3q$.

In characteristic three, $h^2+hk+k^2=(h-k)^2$, so $K_c$ is
a convolution matrix on the additive group of the field. At
frequency $\lambda$ its eigenvalue is
\[
 \sum_{u\in\mathbb F_q}\sum_{s\ne0} \chi(s+cu^2/s-\lambda u)
 =G_\eta\eta(c)\sum_{s\ne0}\eta(s) \chi\left(\left(1-\frac{\lambda^2}{4c}\right)s\right) =q\,\eta(\lambda^2-4c),
\]
where $\eta(0)=0$. Hence, $\|K_c\|\le q$. After the constant
remainder is included, the central chart block has norm at most
$4q^{m-1}$ in every odd characteristic.

Finally, the norm of a self-adjoint two-by-two block operator is
at most the larger diagonal-block norm plus the cross-block norm.
Together with \eqref{eq:shifted-boundary-bound}, the preceding
estimates therefore allow
\[
 C_m\le\max\{C_{m-1}+1,10\}+2,
\]
since $2/q\le1$. Starting from $C_1=4$, this gives $C_2\le12$
and $C_m\le3m+6$ for all $m\ge2$, as asserted.
\end{proof}
\Needspace{10\baselineskip}
\section{Fourth moments and the sharp extension estimate}
\label{sec:fourth-moments-sharp-extension}
\label{sec:refined-moment-estimates}

Throughout this section, $q$ is an odd prime power, $m\ge1$, and
$Q$ is a nondegenerate quadratic form on $V=\F^{2m}$. Both Witt
types are allowed. Fix $j\ne0$ and put $S_j=\{Q=j\}$. Recall that
$|S_j|=q^{2m-1}-\epsilon_Qq^{m-1}\asymp q^{2m-1}$, where
$\epsilon_Q\in\{1,-1\}$ distinguishes the two Witt types.
All unmarked norms use counting measure. We write
\[
 Tf(x)=\sum_{\xi\in S_j}f(\xi)\chi(B(x,\xi)),
 \qquad \mathcal F=q^{-m}T,
\]
where functions on $S_j$ are extended by zero before applying the
unitary ambient Fourier transform $\mathcal F$. Constants may depend
on $m$, but not on $q$, $Q$, or $j$.

Theorem~\ref{thm:completed-generic-blocks} bounds the centered
matrices $M_t$ when $Q(t)\ne j$. To handle $t\in S_j$, we first
prove a centered estimate for the zero cone. The centering is
essential in higher dimensions.

\subsection{A centered sphere--zero-cone estimate}
\label{subsec:mixed-sphere-zero-cone}

The estimate needed in higher dimensions is centered at
$q^{-1}$, the average density of the quadratic levels. This removes a constant contribution
that is too large for the uncentered four-dimensional argument.
We first prove two polarity estimates for auxiliary quadratic spaces.

\Needspace{14\baselineskip}
\begin{lemma}
\label{lem:projective-null-polarity}
Let $U$ be a nondegenerate quadratic space over $\F_q$ of dimension $D\ge2$, with quadratic form $Q_U$ and associated bilinear form $B_U$.
Let $\mathcal N$ be the set of
one-dimensional subspaces on which $Q_U$ vanishes, and define matrices on
$\ell^2(\mathcal N)$ by
\[
 H([u],[v])=\1_{\{B_U(u,v)=0\}},\qquad J([u],[v])=1.
\]
In particular, the diagonal entries of $H$ are one. Then
\begin{equation}\label{eq:projective-null-polarity-bound}
 \|H-q^{-1}J\|_{2\to2}\le q^{\lfloor(D-2)/2\rfloor}.
\end{equation}
If $\mathcal N$ is empty, the matrix is the zero operator.
\end{lemma}

\begin{proof}
We first record the required point counts. Write $Q_U(x)=x^{\mathsf T}A_Ux$
in a basis of $U$. Additive orthogonality and the quadratic Gauss sum
$G_\eta=\sum_{t\in\F_q}\chi(t^2)$ give
\[
 |\{x:Q_U(x)=0\}|
 =q^{D-1}+q^{-1}\eta(\det A_U)G_\eta^D
       \sum_{s\ne0}\eta(s)^D,
 \qquad G_\eta^2=\eta(-1)q.
\]
Consequently, the number of projective null points is
\begin{align*}
n_{2s+1}=\frac{q^{2s}-1}{q-1}, \quad
 n_{2s}=\frac{q^{2s-1}-1}{q-1}+\epsilon_U q^{s-1},
 \quad \epsilon_U=\eta((-1)^s\det A_U)\in\{1,-1\}.
\end{align*}
Here $\epsilon_U=1$ denotes the split form in even dimension. These
counts depend only on the dimension and, in even dimension, the Witt
type. Put $n_0=n_1=0$ when a quotient of dimension zero or one occurs.

First suppose that the Witt index is at least two.
For a null point $[u]$, the quotient $u^\perp/\F_q u$ is nondegenerate
of dimension $D-2$. Each projective null point in this quotient has
$q$ projective lifts to $u^\perp$, in addition to the point $[u]$.
Thus, every row sum of $H$ equals $k=1+q n_{D-2}$.
The subscript refers to the corresponding quadratic quotient; removing
a hyperbolic plane preserves $\epsilon_U$ in even dimension.

For distinct null points $[u],[v]$ with $B_U(u,v)\ne0$, their span is a
hyperbolic plane. The number of null points perpendicular to both is
therefore $b=n_{D-2}$.
If $B_U(u,v)=0$, their span is a totally isotropic plane. Its orthogonal
complement has this plane as its radical, with $q+1$ projective points
in the radical. Each null point of the nondegenerate quotient has
$q^2$ projective lifts. In this case the number is $a=q+1+q^2 n_{D-4}$.
These counts, including the diagonal entries, prove the matrix identity
\begin{equation}\label{eq:null-polarity-square}
 H^2=(k-a)I+(a-b)H+bJ.
\end{equation}

For $D=2s+1$, substitution gives
\[
 a=b,\qquad k-a=q^{2s-2},\qquad k-\frac{|\mathcal N|}q=-\frac1q.
\]
Since $H$ is real symmetric and has constant row sum, its eigenvalues
on the orthogonal complement of the constants belong to
$\{q^{s-1},-q^{s-1}\}$. On constants, $H-q^{-1}J$ has eigenvalue
$-q^{-1}$. This proves \eqref{eq:projective-null-polarity-bound} in
odd dimension.

For $D=2s$, the corresponding identities are
\[
 a-b=\epsilon_U(q-1)q^{s-2},\qquad
 k-a=q^{2s-3},\qquad
 k-\frac{|\mathcal N|}q=\epsilon_U(q-1)q^{s-2}-\frac1q.
\]
Equation~\eqref{eq:null-polarity-square} implies that each nonconstant
eigenvalue belongs to $\{\epsilon_U q^{s-1},-\epsilon_U q^{s-2}\}$.
The centered constant eigenvalue also has absolute value at most
$q^{s-1}$, proving the required estimate.

It remains to consider Witt index at most one. An anisotropic binary
form has no null points. A split binary form has two null points and
$H=I$. A nondegenerate ternary form has $q+1$ null points and $H=I$.
Finally, a nonsplit four-dimensional form has $q^2+1$ null points and
$H=I$: two distinct perpendicular null points would span a totally
isotropic plane. In each nonempty case the centered eigenvalues are
$1$ on the orthogonal complement of the constants and $1-|\mathcal N|/q$ on
constants. Their absolute values satisfy
\eqref{eq:projective-null-polarity-bound} directly.
\end{proof}

\begin{lemma}
\label{lem:odd-cone-polarity}
Let $U$ be nondegenerate of dimension $2k -1$, where $k \ge1$, with quadratic form $Q_U$ and associated bilinear form $B_U$. Put
$C=\{u\in U:Q_U(u)=0\}$, including the zero vector. The matrix
\[
 H_C(u,v)=\1_{\{B_U(u,v)=0\}}-q^{-1},\qquad u,v\in C,
\]
satisfies
\begin{equation}\label{eq:odd-cone-polarity-norm}
 \|H_C\|_{2\to2}=(q-1)q^{k-2}.
\end{equation}
\end{lemma}

\begin{proof}
For $k =1$, the cone consists only of zero and the assertion is
$\|H_C\|=1-q^{-1}$. Suppose that $k \ge2$. The cone has $|C|=q^{2k-2}$ points. Each projective null point has $q-1$ nonzero vector representatives. The matrix $H_C$ annihilates functions
that vanish at the origin and have sum zero on each set of
representatives. On the orthogonal complement of these functions,
use the normalized indicator of each set of representatives. The
block on the nonzero vectors is then $(q-1)(H-q^{-1}J)$,
where $H$ is the projective polarity matrix in dimension $2k -1$.
On the orthogonal complement of the projective constants,
Lemma~\ref{lem:projective-null-polarity} bounds its eigenvalues in
absolute value by $(q-1)q^{k -2}$.

The remaining two-dimensional space is spanned by the point mass at
zero and the normalized constant function on $C\setminus\{0\}$.
In this orthonormal basis, the matrix is
\[
 \begin{pmatrix}
  1-q^{-1}&(1-q^{-1})\sqrt{|C|-1}\\
  (1-q^{-1})\sqrt{|C|-1}&-(1-q^{-1})
 \end{pmatrix}.
\]
For the lower-right entry, we used the fact that $H-q^{-1}J$ has constant eigenvalue $-q^{-1}$ in odd dimension, as shown in the preceding proof. Since $|C|=q^{2k -2}$, this block has eigenvalues
$\pm(q-1)q^{k -2}$. This proves the equality in
\eqref{eq:odd-cone-polarity-norm}.
\end{proof}

\begin{theorem}
\label{thm:centered-sphere-zero-cone}
Let $d=2m\ge2$, let $Q$ be a nondegenerate quadratic form on
$V=\F_q^{2m}$, and put $S_j=\{x:Q(x)=j\}$ for $j\ne0$ and
$S_0=\{x:Q(x)=0\}$. For arbitrary complex functions $f,h$ supported
on $S_j$,
\begin{equation}\label{eq:centered-sphere-cone-convolution}
 \|\1_{S_0}(f*h)\|_2^2-q^{-1}\|f*h\|_2^2
 \le q^{m-1}\|f\|_2^2\|h\|_2^2.
\end{equation}
Equivalently, one has the correlation estimate
\begin{equation}\label{eq:centered-null-correlation}
 \sum_{z\in S_0}
   \left|\sum_y f(y)\overline{h(y+z)}\right|^2
 -q^{-1}\|f*h\|_2^2
 \le q^{m-1}\|f\|_2^2\|h\|_2^2.
\end{equation}
\end{theorem}

\begin{proof}
For $t\in V$, define
\[
 I_t=\{x:Q(x)=Q(x-t)=j\},\qquad
 a_t(x)=f(x)\overline{h(x-t)}.
\]
Expanding the squared convolution and changing variables by
$t=x+y-z$ yields the exact identity
\begin{equation}\label{eq:centered-cone-product-identity}
\|\1_{S_0}(f*h)\|_2^2-q^{-1}\|f*h\|_2^2
=\sum_{t\in V}\sum_{x,y\in I_t}
 \bigl(\1_{\{Q(x+y-t)=0\}}-q^{-1}\bigr)
 a_t(x)\overline{a_t(y)}.
\end{equation}
Indeed, the product before this change of variables is
$f(x)h(z-x)\overline{f(y)h(z-y)}$, and afterwards it is
$a_t(x)\overline{a_t(y)}$. Also,
\begin{equation}\label{eq:centered-cone-correlation-mass}
 \sum_{t\in V}\|a_t\|_2^2=\|f\|_2^2\|h\|_2^2.
\end{equation}
We will bound every matrix in
\eqref{eq:centered-cone-product-identity} by $q^{m-1}$ in operator
norm.

Write $x=t/2+u$ and $y=t/2+v$. The conditions defining $I_t$ become
\[
 u,v\in t^\perp,\qquad Q(u)=Q(v)=a,
 \qquad a=j-\frac{Q(t)}4.
\]
On this index set, $Q(u+v)=0$ is equivalent to $B(u,v)=-a$.
Reflection $v\mapsto-v$ is a permutation of the index set and
preserves operator norms. It is therefore enough to estimate the
matrix
\begin{equation}\label{eq:centered-tangent-matrix}
 \1_{\{B(u,v)=a\}}-q^{-1}.
\end{equation}
There are four cases.

If $t=0$, then $a=j\ne0$. The map $u\mapsto[u:1]$ embeds $S_j$
into the projective null points of the nondegenerate form
\[
 (u,z)\longmapsto Q(u)-jz^2
\]
in dimension $2m+1$. Orthogonality of the images is precisely
$B(u,v)=j$. Thus, \eqref{eq:centered-tangent-matrix} is a principal
compression of the centered projective polarity matrix.
Lemma~\ref{lem:projective-null-polarity} bounds its norm by
$q^{m-1}$.

If $Q(t)\ne0$ and $a\ne0$, the space $t^\perp$ is nondegenerate
of dimension $2m-1$. Apply the same embedding to the form
\[
 (u,z)\longmapsto Q|_{t^\perp}(u)-az^2.
\]
Its dimension is $2m$, so
Lemma~\ref{lem:projective-null-polarity} again gives the bound
$q^{m-1}$, for either Witt type of the enlarged space.

If $Q(t)=4j$, then $a=0$ and $t^\perp$ is nondegenerate of dimension
$2m-1$. Lemma~\ref{lem:odd-cone-polarity} gives the bound
$(q-1)q^{m-2}\le q^{m-1}$ directly.

Finally, suppose that $t\ne0$ and $Q(t)=0$. The radical of
$t^\perp$ is $\F_q t$, and $W=t^\perp/\F_q t$
is nondegenerate of dimension $2m-2$. Here $a=j\ne0$.
Each point of the nonzero level in $W$ has exactly $q$ lifts to
$t^\perp$, and both $Q$ and $B$ are constant on these fibers.
Consequently, the matrix in \eqref{eq:centered-tangent-matrix} is,
after ordering its indices, a tensor product of $J_q$ and the
corresponding centered tangent matrix on the quotient sphere.
For $m\ge2$, embedding that sphere into the projective null points
of a form of dimension $2m-1$ bounds the latter matrix by
$q^{m-2}$. Since $\|J_q\|=q$, the desired bound follows.
For $m=1$, the quotient has dimension zero and its nonzero level is
empty.

All four cases also cover dimension two: the first two use the
ternary and binary cases of Lemma~\ref{lem:projective-null-polarity},
the third uses the one-point cone, and the fourth is empty.
Combining the matrix bounds with
\eqref{eq:centered-cone-product-identity} and
\eqref{eq:centered-cone-correlation-mass} proves
\eqref{eq:centered-sphere-cone-convolution}.

To obtain \eqref{eq:centered-null-correlation}, define
$\widetilde h(x)=\overline{h(-x)}$. Since $S_j=-S_j$, this function
is also supported on $S_j$, and
\[
 (f*\widetilde h)(-z)=\sum_y f(y)\overline{h(y+z)}.
\]
For the Fourier transform associated with the nondegenerate pairing
$B$, one has
$\widehat{\widetilde h}=\overline{\widehat h}$.
Plancherel therefore gives
$\|f*\widetilde h\|_2=\|f*h\|_2$.
Apply \eqref{eq:centered-sphere-cone-convolution} to $f,\widetilde h$
and use $S_0=-S_0$. The same involution proves the converse.
\end{proof}

\Needspace{12\baselineskip}
\subsection{Exceptional translations and the fourth moment}
\label{subsec:exceptional-translations-positive-compression}

\begin{lemma}\label{lem:exceptional-regular}
Let $t\in S_j$, and put
\[
 R_t=\{x\in S_j:Q(x-t)\ne0\}.
\]
For the centered matrix $M_t$ defined in
\eqref{eq:centered-shell-matrix},
\[
 \|\1_{R_t}M_t\1_{R_t}\|_{2\to2}\lesssim q^{m-1}.
\]
The implied constant is absolute.
\end{lemma}

\begin{proof}
Write $u=x-t$. The identities
\eqref{eq:tangent-shell-polarity} give
\[
 Q(u)+2B(t,u)=0,\qquad
 Q(x+y-t)=j+2B(x-t,y-t).
\]
On every projective direction with $Q(u)\ne0$, the first equation
has at most one nonzero solution. Thus, $x\mapsto[x-t]$ is injective
on $R_t$, and the restricted incidence matrix is a principal
submatrix of the polarity matrix $H$ on $\mathbb P^{2m-1}(\F)$,
whose entries are $\1_{\{B(u,v)=0\}}$.

Put
\[
 n=\frac{q^{2m}-1}{q-1},\qquad
 k=\frac{q^{2m-1}-1}{q-1}.
\]
Counting projective hyperplanes and their intersections gives
\[
 H\mathbf1=k\mathbf1,\qquad
 H^2=q^{2m-2}I+\frac{q^{2m-2}-1}{q-1}J,
\]
where $J$ is the all-ones matrix. These formulas include $m=1$:
in that case $H$ is a permutation matrix and $H^2=I$.
On the orthogonal complement of the constants, $H$ has norm
$q^{m-1}$. If $\rho=|S_j|/q^{2m}$, its centered constant
eigenvalue is
\[
 k-\rho n=-q^{-1}+\epsilon_Qq^{-m-1}n,
\]
whose modulus is $O(q^{m-1})$. Hence,
$\|H-\rho J\|\lesssim q^{m-1}$. Compression to the indicated
projective points proves the assertion.
\end{proof}

\begin{proposition}\label{prop:refined-fourth-moment}
For every complex function $f$ supported on $S_j$,
\begin{equation}\label{eq:refined-fourth-moment}
 \|f*f\|_2^2
 \lesssim_m q^{m-1}\|f\|_2^4+
 q^{-1}\|Tf\|_\infty^2\|f\|_2^2.
\end{equation}
Equivalently,
\begin{equation}\label{eq:refined-unitary-fourth}
 \|\mathcal Ff\|_4^4
 \lesssim_m q^{-1}\|\mathcal Ff\|_\infty^2\|f\|_2^2
             +q^{-m-1}\|f\|_2^4.
\end{equation}
\end{proposition}

\begin{proof}
For an arbitrary ambient function $g$, write
\[
 g=g_0+g_1,\qquad g_0=g\1_{S_0},\qquad g_1=g-g_0.
\]
Apply the product-vector identity
\eqref{eq:product-block-expansion} to $f,g_1$. If $t\in S_j$,
then its vector
\[
 a_t(x)=f(x)\overline{g_1(x-t)}
\]
is supported on $R_t$. The generic bound from
Theorem~\ref{thm:completed-generic-blocks},
Lemma~\ref{lem:exceptional-regular}, and
\eqref{eq:product-vector-mass} therefore give
\[
 \|\1_{S_j}(f*g_1)\|_2^2
 \le \frac{|S_j|}{q^{2m}}\|f*g_1\|_2^2
          +C_mq^{m-1}\|f\|_2^2\|g_1\|_2^2.
\]
Convolution by $f$ has Fourier multiplier $Tf$. Since
$|S_j|/q^{2m}\lesssim q^{-1}$, it follows that
\[
 \|\1_{S_j}(f*g_1)\|_2^2
 \lesssim_m
 \bigl(q^{-1}\|Tf\|_\infty^2+q^{m-1}\|f\|_2^2\bigr)
 \|g_1\|_2^2.
\]

For $g_0$, use the correlation consequence of
Theorem~\ref{thm:centered-sphere-zero-cone}. Namely, for every $h$
supported on $S_j$, with $\widetilde h(x)=\overline{h(-x)}$,
\[
 \sum_{z\in S_0}
 \left|\sum_yf(y)\overline{h(y+z)}\right|^2
 \le q^{-1}\|f*\widetilde h\|_2^2
          +q^{m-1}\|f\|_2^2\|h\|_2^2
 \le\bigl(q^{-1}\|Tf\|_\infty^2+q^{m-1}\|f\|_2^2\bigr)
          \|h\|_2^2.
\]
Here the correlation is $(f*\widetilde h)(-z)$ and $S_0=-S_0$.
The identity
\[
 \langle\1_{S_j}(f*g_0),h\rangle
 =\sum_{z\in S_0}g_0(z)
             \sum_yf(y)\overline{h(y+z)}
\]
and Cauchy--Schwarz in $z$ consequently imply
\[
 \|\1_{S_j}(f*g_0)\|_2^2
 \le\bigl(q^{-1}\|Tf\|_\infty^2+q^{m-1}\|f\|_2^2\bigr)
          \|g_0\|_2^2.
\]
Combining the last two estimates by the triangle inequality and
Cauchy--Schwarz, and using
$\|g_0\|_2^2+\|g_1\|_2^2=\|g\|_2^2$, proves
\begin{equation}\label{eq:positive-convolution-bound}
 \|\1_{S_j}(f*g)\|_2^2
 \lesssim_m
 \bigl(q^{-1}\|Tf\|_\infty^2+q^{m-1}\|f\|_2^2\bigr)
 \|g\|_2^2.
\end{equation}

Let $P=\mathcal F\1_{S_j}\mathcal F^*$, the orthogonal projection
onto the Fourier transforms of functions supported on $S_j$.
Since
\[
 \mathcal F(f*g)=q^m(\mathcal Ff)(\mathcal Fg),
\]
dividing \eqref{eq:positive-convolution-bound} by $q^{2m}$ gives
\[
 \|PD_{\mathcal Ff}\|_{2\to2}^2
 =\|PD_{|\mathcal Ff|^2}P\|_{2\to2}
 \lesssim_m q^{-1}\|\mathcal Ff\|_\infty^2
                      +q^{-m-1}\|f\|_2^2.
\]
Here $D_u$ denotes multiplication by $u$; the equality follows by
multiplying $PD_{\mathcal Ff}$ by its adjoint. Since
$P\mathcal Ff=\mathcal Ff$, taking its quadratic form at
$\mathcal Ff$ proves \eqref{eq:refined-unitary-fourth}. Finally,
\[
 \|f*f\|_2^2=q^{2m}\|\mathcal Ff\|_4^4,
 \qquad \|\mathcal Ff\|_\infty=q^{-m}\|Tf\|_\infty,
\]
which gives \eqref{eq:refined-fourth-moment}.
\end{proof}

\subsection{An orthogonal decomposition}
\label{subsec:coherent-extraction}

The following decomposition will also be used in the pinned-distance
energy estimates.
We put $r_m=2(m+1)/m$, so that $2<r_m\le4$.

\begin{lemma}\label{lem:coherent-decomposition}
There is an absolute constant $A$ with the following property.
Let $f$ be supported on $S_j$, with counting norm $\|f\|_2=1$,
and define
\[
 a_x(\xi)=|S_j|^{-1/2}\chi(-B(x,\xi)),\qquad \xi\in S_j.
\]
There are distinct points $x_1,\ldots,x_L$ and coefficients $c_i$
such that
\[
 f=h+r,\qquad h=\sum_{i=1}^Lc_i a_{x_i},\qquad
 L\le\frac{|S_j|}{A^2q^m}\lesssim q^{m-1},
\]
\[
 \|h\|_2,\|r\|_2\le1,\qquad
 \sum_i|c_i|^2\lesssim1,\qquad
 \sum_i|c_i|^{r_m}\lesssim1,\qquad
 \|Tr\|_\infty\le Aq^{m/2}.
\]
The remainder satisfies
\begin{equation}\label{eq:coherent-residual-fourth}
 \|\mathcal Fr\|_4^4\lesssim_m q^{-m-1}.
\end{equation}
Moreover, $Th=b+e$, where
\begin{equation}\label{eq:coherent-atomic-main}
 b=\sqrt{|S_j|}\sum_i c_i\1_{\{x_i\}},\qquad
 \|b\|_2^2\lesssim |S_j|,\qquad
 \|b\|_{r_m}^{r_m}\lesssim |S_j|^{r_m/2},
\end{equation}
and
\begin{equation}\label{eq:coherent-atomic-tail}
 \|e\|_\infty\lesssim q^{(m-1)/2},\qquad
 \|e\|_2\lesssim q^m.
\end{equation}
\end{lemma}

\begin{proof}
Lemmas~\ref{lem:size-kernel} and~\ref{lem:kernel-linfty} give
\[
 \|a_x\|_2=1,\qquad
 |\langle a_x,a_y\rangle|\lesssim q^{-m+1/2}\quad(x\ne y),
\]
\[
 Ta_x(x)=\sqrt{|S_j|},\qquad
 |Ta_x(y)|\lesssim1\quad(y\ne x).
\]
The constants here are absolute for all $m\ge1$, both Witt types,
and all odd prime powers.

Start with the zero subspace $W$, and at each stage put
$r=f-P_Wf$. If $\|Tr\|_\infty>Aq^{m/2}$, choose a point $x$
at which this inequality holds and enlarge $W$ by $a_x$.
Previously selected points satisfy
\[
 Tr(x_i)=\sqrt{|S_j|}\langle r,a_{x_i}\rangle=0,
\]
so the selected points are distinct. Put $u=a_x-P_Wa_x$.
Then $u\ne0$, $\|u\|_2\le1$, and the decrease of the squared
remainder norm is
\[
 \frac{|\langle r,u\rangle|^2}{\|u\|_2^2}
 \ge |\langle r,a_x\rangle|^2
 >\frac{A^2q^m}{|S_j|}.
\]
Thus, at most $|S_j|/(A^2q^m)$ points are selected.

The Gram matrix of the selected functions has diagonal entries one
and off-diagonal row sums bounded by
\[
 CLq^{-m+1/2}\le\frac{C'}{A^2\sqrt q}.
\]
Choose $A$ so that this is at most $1/2$ for all odd $q$.
The Gram matrix is then uniformly invertible. At termination,
$h=P_Wf$ and $r=f-h$, so their norms are at most one and
$\sum_i|c_i|^2\lesssim1$. Since $r_m\ge2$,
$\sum_i|c_i|^{r_m}\lesssim1$ as well. The stopping rule gives
$\|Tr\|_\infty\le Aq^{m/2}$. This construction also applies when
$m=1$; the selected family can be empty.

Applying Proposition~\ref{prop:refined-fourth-moment} to $r$ gives
\[
 \|r*r\|_2^2
 \lesssim_m q^{m-1}+q^{-1}(Aq^{m/2})^2
 \lesssim_m q^{m-1}.
\]
The convolution normalization proves
\eqref{eq:coherent-residual-fourth}.

Define $b$ by \eqref{eq:coherent-atomic-main}, and put $e=Th-b$.
The points $x_i$ are distinct, giving the stated norms of $b$.
For each $y$, the term with $x_i=y$ is removed exactly by $b$;
all other terms satisfy the off-diagonal kernel bound. Hence,
\[
 \|e\|_\infty\lesssim\sum_i|c_i|
 \le\sqrt L\left(\sum_i|c_i|^2\right)^{1/2}
 \lesssim q^{(m-1)/2}.
\]
Finally, Plancherel gives $\|Th\|_2\le q^m$, while
$\|b\|_2\lesssim\sqrt{|S_j|}\lesssim q^{m-1/2}$.
Their sum proves the bound for $\|e\|_2$.
\end{proof}

\Needspace{5\baselineskip}
\subsection{The strong endpoint and its sharpness}
\label{subsec:strong-endpoint-three}

\begin{proof}[Proof of Theorem~\ref{thm:main-endpoint-three}]
Normalize the counting norm of $f$ on $S_j$ to one, and use
Lemma~\ref{lem:coherent-decomposition}. For the remainder,
interpolation between $L^2$ and $L^4$ gives
\[
 \|\mathcal Fr\|_{r_m}^{r_m}
 \le \|\mathcal Fr\|_2^{4-r_m}
       \bigl(\|\mathcal Fr\|_4^4\bigr)^{(r_m-2)/2}
 \lesssim_m q^{-(m+1)/m}=q^{-r_m/2}.
\]
For $m=1$ this is simply the fourth-moment estimate, so no
interpolation is needed.

For the extracted part, $\mathcal Fh=q^{-m}(b+e)$ and
\[
 q^{-mr_m}\|b\|_{r_m}^{r_m}
 \lesssim q^{-mr_m}|S_j|^{r_m/2}
 \lesssim q^{-r_m/2}.
\]
The bounds for $e$ imply
\[
 q^{-mr_m}\|e\|_{r_m}^{r_m}
 \le q^{-mr_m}\|e\|_\infty^{r_m-2}\|e\|_2^2 \lesssim q^{-mr_m+(m-1)/m+2m}
 =q^{-r_m/2}.
\]
The triangle inequality and homogeneity give
\[
 \|\mathcal Ff\|_{r_m}\lesssim_m q^{-1/2}\|f\|_2.
\]
Since
\[
 (f\,d\sigma_j)^\vee=\frac{q^m}{|S_j|}\mathcal Ff,
 \qquad
 \|f\|_{L^2(S_j,d\sigma_j)}=|S_j|^{-1/2}\|f\|_2,
\]
we obtain
\[
 \|(f\,d\sigma_j)^\vee\|_{r_m}
 \lesssim_m
 \frac{q^{m-1/2}}{\sqrt{|S_j|}}
 \|f\|_{L^2(S_j,d\sigma_j)}
 \lesssim_m\|f\|_{L^2(S_j,d\sigma_j)}.
\]
The comparison is uniform because
$|S_j|/q^{2m-1}=1-\epsilon_Qq^{-m}\ge2/3$.
Counting-measure monotonicity gives the estimate for every
$r\ge r_m$, and duality gives its restriction formulation.

For sharpness, choose $x_0\in S_j$. The orthogonal complement
$x_0^\perp$ is a nondegenerate quadratic space of dimension
$2m-1$, so it contains a totally isotropic subspace $W$ of
dimension $m-1$. This statement includes $W=\{0\}$ when $m=1$.
Since $x_0\perp W$, the affine subspace $x_0+W$ lies in $S_j$.
For $f=\1_{x_0+W}$, additive orthogonality shows that its extension
has modulus $q^{m-1}/|S_j|$ on $W^\perp$, of size $q^{m+1}$,
and vanishes elsewhere. Thus,
\[
 \frac{\|(f\,d\sigma_j)^\vee\|_r}
      {\|f\|_{L^2(S_j,d\sigma_j)}}
 =\frac{q^{(m+1)/r+(m-1)/2}}{\sqrt{|S_j|}}
 \asymp q^{(m+1)/r-m/2}.
\]
Uniform boundedness requires $r\ge2(m+1)/m$. This proves
sharpness for both Witt types in every even dimension.
\end{proof}
\section{The pinned distance theorem}
\label{subsec:optimized-pinned-distances}

Throughout this section, $Q$ is any nondegenerate quadratic form on
$\F^{2m}$, and $B$ is its associated bilinear form. We use the unnormalized
Fourier transform with phase $\chi(-B(x,\xi))$ and normalized surface
measure. Thus, the same form describes distances and Fourier level
sets. Put $\epsilon_Q=1$ when $Q$ is split and $\epsilon_Q=-1$
otherwise. Constants may depend on $m$, but not on $Q$ or $q$.
For $E,F\subseteq\F^{2m}$ and $t\in\F$, define
\[
 \nu_{E,F}(t)=\#\{(x,y)\in E\times F:Q(x-y)=t\}.
\]
Then $\sum_t\nu_{E,F}(t)=|E||F|$.

The quadratic Gauss-sum identity gives
\begin{equation}\label{eq:zero-sphere-fourier}
 \widehat{\1_{S_0}}(\xi)
 =q^{2m-1}\1_{\{0\}}(\xi)
   +\epsilon_Q q^m\1_{S_0}(\xi)-\epsilon_Q q^{m-1}.
\end{equation}
This follows by writing the indicator as
$q^{-1}\sum_{s\in\F}\chi(sQ(x))$ and completing the square for
$s\ne0$. The quadratic Gauss sum for $s\ne0$ is $\epsilon_Q q^m$,
independent of $s$ because the dimension is even. In particular,
\[
 |S_j|=q^{2m-1}-\epsilon_Q q^{m-1}\asymp q^{2m-1},
 \qquad j\ne0.
\]

\begin{lemma}\label{lem:zero-distance-incidence}
For all $E,F\subseteq\F^{2m}$,
\begin{equation}\label{eq:zero-incidence}
 \nu_{E,F}(0)
 \leq\frac{|E||F|}{q}+2q^m|E|^{1/2}|F|^{1/2}.
\end{equation}
\end{lemma}
\begin{proof}
Fourier inversion and \eqref{eq:zero-sphere-fourier} yield
\[
 \nu_{E,F}(0)=\frac{|E||F|}{q}
 +\epsilon_Qq^{-m}\sum_{\xi\in S_0}
       \widehat{\1_E}(\xi)\overline{\widehat{\1_F}(\xi)}
 -\epsilon_Qq^{m-1}|E\cap F|.
\]
Cauchy--Schwarz and Plancherel bound the absolute value of the
middle term by $q^m|E|^{1/2}|F|^{1/2}$. The last term has absolute
value at most $q^{m-1}|E|^{1/2}|F|^{1/2}$, proving the claim.
\end{proof}

For fixed $E\subseteq\F^{2m}$, define
\[
\Psi_j(y)=\sum_{\xi\in S_j}\widehat{\1_E}(\xi)
                  \chi(B(y,\xi)),
 \qquad
 \nu_y(j)=|\{x\in E:Q(x-y)=j\}|.
\]
The next identity is independent of the Witt type. An averaged
version appears in \cite[Section~5]{chapman}; we include the proof
to specify the normalization and the contribution of radius zero.

\begin{lemma}\label{lem:exact-pinned-identity}
For every $y\in\F^{2m}$,
\begin{align*}
 \sum_{j\in\F^\times}\nu_y(j)^2
 = \frac{|E|^2}{q}
   +q^{-2m}\sum_{j\in\F^\times}|\Psi_j(y)|^2 \bigl(q^{-2m}|\Psi_0(y)|^2-\nu_y(0)^2
                    -q^{2m-1}\1_E(y)\bigr).
\end{align*}
\end{lemma}
\begin{proof}
Set $U_s(y)=\sum_{x\in E}\chi(sQ(x-y))$. Orthogonality gives
\[
 \sum_{t\in\F}\nu_y(t)^2
 =\frac{|E|^2}{q}+\frac1q\sum_{s\ne0}|U_s(y)|^2.
\]
Completing the square in the Fourier inversion formula gives, for
$s\ne0$,
\[
 U_s(y)=\epsilon_Qq^{-m}\sum_{\xi\in\F^{2m}}
 \widehat{\1_E}(\xi)\chi(B(y,\xi))
                    \chi\left(-\frac{Q(\xi)}{4s}\right).
\]
The factor $\epsilon_Q$ disappears after taking the squared absolute value. Since
\[
 \sum_{s\ne0}\chi\left(\frac{a}{4s}\right)
 =q\1_{\{a=0\}}-1,
\]
expansion and grouping frequencies according to their $Q$-values
show that
\[
 \frac1q\sum_{s\ne0}|U_s(y)|^2
 =q^{-2m}\sum_{j\in\F}|\Psi_j(y)|^2
  -q^{-2m-1}\left|\sum_{\xi\in\F^{2m}}
        \widehat{\1_E}(\xi)\chi(B(y,\xi))\right|^2.
\]
The sum in the last term equals $q^{2m}\1_E(y)$.
Subtracting $\nu_y(0)^2$ proves the identity.
\end{proof}

\begin{lemma}\label{lem:zero-shell}
For every $E\subseteq\F^{2m}$ and $y\in\F^{2m}$,
\[
 q^{-2m}|\Psi_0(y)|^2-\nu_y(0)^2-q^{2m-1}\1_E(y)
 \leq\frac{3|E|^2}{q^2}.
\]
For split $Q$, the constant $3$ can be replaced by $1$.
\end{lemma}
\begin{proof}
Write $v=\nu_y(0)$, and $I=\1_E(y)$. By
\eqref{eq:zero-sphere-fourier},
\[
 \Psi_0(y)=\epsilon_Qq^m v-\epsilon_Qq^{m-1}|E|
                      +q^{2m-1}I.
\]
Consequently, the left-hand side is
\begin{equation}\label{eq:zero-shell-expanded-even}
 \frac{|E|^2}{q^2}-\frac{2|E|v}{q}
 +I\left(2\epsilon_Qq^{m-1}v
       -2\epsilon_Qq^{m-2}|E|-(q-1)q^{2m-2}\right).
\end{equation}
For $I=0$, the claimed bound with constant $1$ is immediate.
Suppose $I=1$.

If $\epsilon_Q=1$ and $|E|\geq q^m$, the coefficient of $v$ after
combining its two terms is nonpositive, and all remaining terms
apart from $|E|^2/q^2$ are nonpositive. If $|E|<q^m$, use $v\leq |E|$
and discard the term $-2q^{m-2}|E|$. The terms apart from $|E|^2/q^2$
are at most
\[
 2q^{m-1}|E|-\frac{2|E|^2}{q}-(q-1)q^{2m-2}
 \leq\frac{q^{2m-1}}2-(q-1)q^{2m-2}\leq0,
\]
since $q\geq3$. This proves the split assertion.

If $\epsilon_Q=-1$, the terms containing $v$ are nonpositive.
For $|E|<q^m$, the remaining bracket in
\eqref{eq:zero-shell-expanded-even} is at most
$(3-q)q^{2m-2}\leq0$. For $|E|\geq q^m$, its positive term satisfies
$2q^{m-2}|E|\leq2|E|^2/q^2$. This proves the stated bound for the
nonsplit form as well.
\end{proof}

Put
\[
 T_y=\sum_{j\ne0}\nu_y(j)^2,
 \qquad
 \mathcal I^\times(E;F)=\sum_{y\in F}T_y.
\]
Thus, $\mathcal I^\times(E;F)$ counts triples $(x,x',y)\in E^2\times F$
with $Q(x-y)=Q(x'-y)\ne0$. The preceding lemmas give
\begin{equation}\label{eq:optimized-pinned-energy-start}
 \mathcal I^\times(E;F)
 \leq\frac{|E|^2|F|}{q}
 +q^{-2m}\sum_{j\ne0}\sum_{y\in F}|\Psi_j(y)|^2
 +\frac{3|E|^2|F|}{q^2}.
\end{equation}

\Needspace{11\baselineskip}
\begin{lemma}\label{lem:almost-every-pin-energy-criterion}
Let $E,F\subseteq\F^{2m}$ be nonempty. Suppose that
$\varepsilon_q,\rho_q\geq0$ tend to zero and that
\begin{equation}\label{eq:almost-every-pin-IEF-hypothesis}
 \mathcal I^\times(E;F)
 \leq(1+\varepsilon_q)\frac{|E|^2|F|}{q},
\end{equation}
\begin{equation}\label{eq:almost-every-pin-zero-hypothesis}
 \frac{\nu_{E,F}(0)}{|E||F|}\leq\rho_q.
\end{equation}
Then there is $F'\subseteq F$ with $|F'|=(1-o(1))|F|$ such that
$|\Delta_y(E)\setminus\{0\}|=(1-o(1))q$ uniformly for $y\in F'$.
\end{lemma}
\begin{proof}
Set $\tau=(\rho_q+q^{-1})^{1/2}=o(1)$, and let $F_0=\{y\in F:\nu_y(0)\leq\tau |E|\}$. By Markov's inequality, $|F\setminus F_0|\leq\tau|F|$.
For $y\in F_0$, Cauchy--Schwarz gives
\[
 T_y\geq\frac{(|E|-\nu_y(0))^2}{q-1}
       \geq(1-2\tau)\frac{|E|^2}{q}.
\]
It follows from \eqref{eq:almost-every-pin-IEF-hypothesis} that
\[
 \sum_{y\in F_0}\max\left\{\frac{qT_y}{|E|^2}-1,0\right\}
 \leq(\varepsilon_q+3\tau)|F|.
\]
With $\delta=(\varepsilon_q+3\tau)^{1/2}=o(1)$, a second application
of Markov's inequality shows that at most $\delta|F|$ points of
$F_0$ satisfy $T_y>(1+\delta)|E|^2/q$. Remove these points to obtain
$F'$. For every remaining point,
\[
 |\Delta_y(E)\setminus\{0\}|
 \geq\frac{(|E|-\nu_y(0))^2}{T_y}
 \geq\frac{(1-\tau)^2}{1+\delta}q=(1-o(1))q.
\]
The bound $|\Delta_y(E)\setminus\{0\}|\leq q-1$
completes the proof.
\end{proof}

For $1\leq M\leq q^{2m}$, define the numerical factor
\[
 K_m(M)=\min\left\{1+Mq^{-m+1/2},
          1+\frac{M}{q^m}+\sqrt{\frac{M}{q^{m-1}}},\ q\right\}.
\]
We next bound pinned energy directly. The three terms come from the
sphere kernel, the orthogonal decomposition, and Plancherel,
respectively.

\begin{proposition} \label{prop:direct-pinned-energy}
For all nonempty $E,F\subseteq\F_q^{2m}$,
\begin{equation}\label{eq:quantitative-pinned-energy}
 \frac{\mathcal I^\times(E;F)}{|E|^2|F|/q}
 \leq1+C_m\frac{q^{2m}K_m(|F|)}{|E||F|}+\frac3q.
\end{equation}
\end{proposition}
\begin{proof}
For $j\ne0$, let
\[
 g_j=\widehat{\1_E}\,\1_{S_j},\qquad a_j=\|g_j\|_2.
\]
The functions $g_j$ have disjoint frequency supports. Thus, Plancherel
implies
\begin{equation}\label{eq:direct-pinned-shell-mass}
 \sum_{j\ne0}a_j^2
 =\sum_{Q(\xi)\ne0}|\widehat{\1_E}(\xi)|^2
 \leq q^{2m}|E|.
\end{equation}
Terms with $a_j=0$ contribute nothing. Otherwise put $f_j=g_j/a_j$,
so that $\|f_j\|_2=1$ in counting measure. With
$\mathcal F=q^{-m}T$ as in Section~\ref{sec:fourth-moments-sharp-extension},
we have $\Psi_j=a_jTf_j$ and therefore
\begin{equation}\label{eq:direct-pinned-shell-energy}
 q^{-2m}\sum_{j\ne0}\|\Psi_j\|_{L^2(F)}^2
 =\sum_{j\ne0}a_j^2\|\mathcal F f_j\|_{L^2(F)}^2.
\end{equation}

Apply Lemma~\ref{lem:coherent-decomposition} separately to every $f_j$.
Write $f_j=h_j+r_j$ and $Th_j=b_j+e_j$. The lemma gives
\[
 \|b_j\|_2^2\lesssim|S_j|,\qquad
 \|e_j\|_\infty^2\lesssim q^{m-1},\qquad
 \|\mathcal F r_j\|_4^4\lesssim_m q^{-m-1}.
\]
The triangle inequality followed by Cauchy--Schwarz on $F$ yields
\begin{align*}
 \|\mathcal F f_j\|_{L^2(F)}^2
 &\lesssim q^{-2m}\|b_j\|_2^2
       +q^{-2m}|F|\|e_j\|_\infty^2
       +|F|^{1/2}\|\mathcal F r_j\|_4^2\\
 &\lesssim_m q^{-1}+|F|q^{-m-1}
                    +|F|^{1/2}q^{-(m+1)/2}.
\end{align*}
Inserting this in \eqref{eq:direct-pinned-shell-energy} and using
\eqref{eq:direct-pinned-shell-mass} proves
\begin{equation}\label{distance_use_later}
 q^{-2m}\sum_{j\ne0}\|\Psi_j\|_{L^2(F)}^2
 \lesssim_m q^{2m-1}|E|
       \left(1+\frac{|F|}{q^m}+\sqrt{\frac{|F|}{q^{m-1}}}\right).
\end{equation}

For the estimate from the sphere kernel, let $A_{F,j}$ be the matrix from
$\ell^2(S_j)$ to $\ell^2(F)$ with entries
$q^{-m}\chi(B(y,\xi))$. Its Gram matrix has entries
\[
 (A_{F,j}A_{F,j}^*)(y,y')
 =q^{-2m}\sum_{\xi\in S_j}\chi(B(y-y',\xi)).
\]
The diagonal entries are $q^{-2m}|S_j|\lesssim q^{-1}$.
Lemma~\ref{lem:kernel-linfty} bounds the absolute value of each off-diagonal entry by
$Cq^{-m-1/2}$. The maximum absolute row sum
therefore bounds the operator norm by
\[
 \|A_{F,j}A_{F,j}^*\|_{2\to2}
 \lesssim q^{-1}+|F|q^{-m-1/2}.
\]
Since $A_{F,j}f_j=(\mathcal Ff_j)|_F$, equations
\eqref{eq:direct-pinned-shell-energy} and
\eqref{eq:direct-pinned-shell-mass} consequently give the alternative
bound $Cq^{2m-1}|E|(1+|F|q^{-m+1/2})$.

Finally, unitarity of the ambient Fourier transform gives
$\|\mathcal Ff_j\|_{L^2(F)}^2\leq\|f_j\|_2^2=1$.
The same two equations then bound the sum in
\eqref{eq:direct-pinned-shell-energy} by
$q^{2m}|E|$. Taking the minimum of these three bounds and inserting
it into \eqref{eq:optimized-pinned-energy-start} proves
\eqref{eq:quantitative-pinned-energy}.
\end{proof}

\Needspace{10\baselineskip}
\begin{corollary}\label{cor_distance_concrete}
Let $\varepsilon>0$ and let $E,F\subseteq\F^{2m}$ be nonempty.
Suppose one of the following conditions holds:
\begin{enumerate}
\item $1\leq|F|\leq q^{m-1/2}$ and
      $|E||F|\geq q^{2m+\varepsilon}$;
\item $q^{m-1/2}\leq|F|\leq q^m$ and
      $|E|\geq q^{m+1/2+\varepsilon}$;
\item $q^m\leq|F|\leq q^{m+1}$ and
      $|E||F|^{1/2}\geq q^{(3m+1)/2+\varepsilon}$;
\item $q^{m+1}\leq|F|\leq q^{2m}$ and
      $|E||F|\geq q^{2m+1+\varepsilon}$.
\end{enumerate}
Then there is $F'\subseteq F$ such that
$|F'|=(1-o_\varepsilon(1))|F|$ and
$|\Delta_y(E)\setminus\{0\}|=(1-o_\varepsilon(1))q$ uniformly
for every $y\in F'$.
\end{corollary}
\begin{proof}
The definition of $K_m$ gives, respectively,
\[
 K_m(|F|)\lesssim
 \begin{cases}
 1,&1\leq |F|\leq q^{m-1/2},\\
 |F|q^{-m+1/2},&q^{m-1/2}\leq |F|\leq q^m,\\
 |F|^{1/2}q^{-(m-1)/2},&q^m\leq |F|\leq q^{m+1},\\
 q,&q^{m+1}\leq |F|\leq q^{2m}.
 \end{cases}
\]
Each hypothesis makes the second term on the right of
\eqref{eq:quantitative-pinned-energy} at most $C_mq^{-\varepsilon}$.
Each also implies $|E||F|\geq q^{2m+\varepsilon}$, so
\eqref{eq:zero-incidence} makes the proportion of zero-distance pairs
tend to zero. Lemma~\ref{lem:almost-every-pin-energy-criterion}
proves the conclusion.
\end{proof}

\begin{proof}[Proof of Corollary~\ref{cor_distance_2}]
Put $r_m=2(m+1)/m$, and
$g_j=\widehat{\1_E}\,\1_{S_j}$. The global endpoint in
Theorem~\ref{thm:main-endpoint-three} gives
\[
 \|\Psi_j\|_{L^{r_m}(V)}
 =|S_j|\|(g_j\,d\sigma_j)^\vee\|_{L^{r_m}(V)}
 \lesssim_m |S_j|^{1/2}\|g_j\|_2.
\]
H\"older's inequality on $F$ and Plancherel consequently imply
\begin{align*}
 q^{-2m}\sum_{j\ne0}\|\Psi_j\|_{L^2(F)}^2
 &\leq q^{-2m}|F|^{1/(m+1)}
                   \sum_{j\ne0}\|\Psi_j\|_{L^{r_m}(V)}^2\\
 &\lesssim_m q^{-2m}|F|^{1/(m+1)}q^{2m-1}
                   \sum_{j\ne0}\|g_j\|_2^2\\
 &\leq C_mq^{2m-1}|E||F|^{1/(m+1)}.
\end{align*}
Substitution in \eqref{eq:optimized-pinned-energy-start} gives
\[
 \frac{\mathcal I^\times(E;F)}{|E|^2|F|/q}
 \leq1+C_m\frac{q^{2m}}{|E||F|^{m/(m+1)}}+\frac3q
 =1+o_\varepsilon(1).
\]
The hypothesis also implies $|E||F|\geq q^{2m+\varepsilon}$, so
\eqref{eq:zero-incidence} controls zero distances. The conclusion
follows from Lemma~\ref{lem:almost-every-pin-energy-criterion}.
The same estimates prove the assertion whenever
$|E||F|^{m/(m+1)}/q^{2m}\to\infty$.
\end{proof}

\begin{proof}[Proof of Corollary~\ref{cor:improved-middle-pins}]
This is case \textup{(3)} of Corollary~\ref{cor_distance_concrete}.
\end{proof}

\begin{proof}[Proof of Corollary~\ref{cor:same-set-seven-thirds}]
Put $F=E$. By the hypothesis, $|E|\geq q^m$ for all
sufficiently large $q$. The definition of $K_m$ gives
\[
 K_m(|E|)\lesssim\sqrt{\frac{|E|}{q^{m-1}}}
 \qquad(q^m\leq |E|\leq q^{2m}).
\]
For $|E|\leq q^{m+1}$ this follows from the second term defining
$K_m$, and for $|E|\geq q^{m+1}$ it follows from $K_m(|E|)\leq q$.
The direct energy bound \eqref{eq:quantitative-pinned-energy} gives
\[
 \frac{\mathcal I^\times(E;E)}{|E|^3/q}
 \leq1+C_m\frac{q^{(3m+1)/2}}{|E|^{3/2}}+\frac3q
 =1+o(1).
\]
Also, \eqref{eq:zero-incidence} gives
\[
 \frac{\nu_{E,E}(0)}{|E|^2}\leq q^{-1}+\frac{2q^m}{|E|}=o(1).
\]
Lemma~\ref{lem:almost-every-pin-energy-criterion} completes the proof.
The error is uniform whenever $|E|/q^{m+1/3}\geq\omega(q)$ for a
specified function $\omega(q)\to\infty$, and in particular is
$o_\varepsilon(1)$ under the hypothesis
$|E|\geq q^{m+1/3+\varepsilon}$.
\end{proof}

\Needspace{10\baselineskip}
\section{A localized extension conjecture over prime fields}
\label{section_app}

The global restriction and distance theorems above hold over all odd
finite fields and for both Witt types. We now discuss a stronger
concentration problem over prime fields. Throughout this section,
$p$ is an odd prime, $V=\F_p^{2m}$, and $Q$ is a nondegenerate
quadratic form on $V$. We retain the Fourier transform defined by
its associated bilinear form $B$ and normalized surface measure on
$S_j=\{Q=j\}$, where $j\ne0$.

\subsection{Localized norms and bounds} \label{sec:proof-thm-main_conjecture}

For $0\le\vartheta\le2m$, define $\Lambda_{S_j}(\vartheta)$ to be
the infimum of all $\lambda\ge0$ such that
\begin{equation}\label{eq:localized-horizontal-extension}
 \|(f\,d\sigma_j)^\vee\|_{L^2(F)}
 \le p^\lambda\|f\|_{L^2(S_j,d\sigma_j)}
\end{equation}
for every $F\subseteq V$ with $1\le|F|\le p^\vartheta$ and every
complex function $f$ on $S_j$. By duality, this is equivalent to
\begin{equation}\label{eq:localized-horizontal-restriction}
 \|\widehat g\|_{L^2(S_j,d\sigma_j)}
 \le p^\lambda\|g\|_{L^2(V)}
 \qquad(\operatorname{supp}g\subseteq F)
\end{equation}
for every such $F$ and $g$. For a fixed $F$, the two operators in
these inequalities are adjoints with respect to counting measure on
$F$ and normalized measure on $S_j$. The infima are attained, since
there are finitely many sets $F$ and each operator acts between
finite-dimensional Hilbert spaces. The notation suppresses the
dependence on $p$ and $Q$.

The decomposition used to prove the global endpoint also gives the
following bound.

\begin{theorem} \label{thm_main_Lambda}
For every $F\subseteq V$ and every complex function
$f$ on $S_j$,
\begin{equation}\label{eq:localized-four-range-norm}
 \begin{split}
 \|(f\,d\sigma_j)^\vee\|_{L^2(F)}^2
 \lesssim_m
 \min\left\{1+|F|p^{-m+1/2},\,
       1+\frac{|F|}{p^m}+\sqrt{\frac{|F|}{p^{m-1}}},\,p\right\}
 \|f\|_{L^2(S_j,d\sigma_j)}^2.
 \end{split}
\end{equation}
Consequently, for every $\varepsilon>0$ and all sufficiently large
odd primes $p$, uniformly over $Q$, $j\ne0$, and
$0\le\vartheta\le2m$,
\[
 \Lambda_{S_j}(\vartheta)
 \le\Lambda_{\sharp,m}(\vartheta)+\varepsilon,
\]
where
\begin{equation}\label{eq:Lambda-sharp-definition}
 \Lambda_{\sharp,m}(\vartheta)=
 \begin{cases}
 0,&0\le\vartheta\le m-\tfrac12,\\[1mm]
 (2\vartheta-2m+1)/4,&m-\tfrac12\le\vartheta\le m,\\[1mm]
 (\vartheta-m+1)/4,&m\le\vartheta\le m+1,\\[1mm]
 1/2,&m+1\le\vartheta\le2m.
 \end{cases}
\end{equation}
\end{theorem}

\begin{proof}
The assertion is immediate if $F$ is empty or $f=0$. To prove the
middle bound, normalize the counting $L^2$ norm of $f$ to one. Write
$T f(x)=\sum_{\xi\in S_j}f(\xi)\chi(B(x,\xi))$, so that
$\mathcal Ff=p^{-m}Tf$. By
Lemma~\ref{lem:coherent-decomposition}, write $f=h+g$ and $Th=b+e$,
where
\[
 \|\mathcal Fg\|_4^4\lesssim_m p^{-m-1},\qquad
 \|b\|_2^2\lesssim |S_j|,\qquad
 \|e\|_\infty\lesssim p^{(m-1)/2}.
\]
Cauchy--Schwarz on $F$ gives
\[
 \begin{aligned}
 \|\mathcal Ff\|_{L^2(F)}^2
 &\lesssim p^{-2m}\|b\|_2^2
        +p^{-2m}|F|\|e\|_\infty^2
        +|F|^{1/2}\|\mathcal Fg\|_4^2\\
 &\lesssim_m p^{-1}+|F|p^{-m-1}
                     +|F|^{1/2}p^{-(m+1)/2}.
 \end{aligned}
\]
Since the counting $L^2$ norm of $f$ is one, the ratio of its squared
extension and surface norms is $p^{2m}/|S_j|$ times
$\|\mathcal Ff\|_{L^2(F)}^2$. The relation $|S_j|\asymp p^{2m-1}$ proves the middle
bound in \eqref{eq:localized-four-range-norm}, and homogeneity
removes the normalization.

For the first bound, let $g$ be supported on $F$.
Lemma~\ref{lem:global-auxiliary-bounds} and Cauchy--Schwarz give
\[
 \|\widehat g\|_{L^2(S_j,d\sigma_j)}^2
 \lesssim(1+|F|p^{-m+1/2})\|g\|_2^2.
\]
Duality proves the first bound. Finally,
the second estimate in Lemma~\ref{lem:global-auxiliary-bounds} yields
\[
 \|\widehat g\|_{L^2(S_j,d\sigma_j)}^2
 \le\frac{p^{2m}}{|S_j|}\|g\|_2^2\lesssim p\|g\|_2^2,
\]
which proves the last bound by duality.

If $|F|\le p^\vartheta$, the first term in the minimum is bounded
by a constant for $\vartheta\le m-1/2$ and by
$Cp^{\vartheta-m+1/2}$ for $m-1/2\le\vartheta\le m$. When
$m\le\vartheta\le m+1$, the middle term is at most
$Cp^{(\vartheta-m+1)/2}$. For $\vartheta\ge m+1$, use the last
term. Taking square roots gives the bound in
\eqref{eq:Lambda-sharp-definition}. A single constant $C_m$ works
throughout the range, and $C_m\le p^\varepsilon$ for all sufficiently
large $p$.
\end{proof}

For comparison, the global endpoint and H\"older's inequality give
the squared operator norm bound $|F|^{1/(m+1)}$. The inequality
\[
 2\Lambda_{\sharp,m}(\vartheta)\le\frac{\vartheta}{m+1}
 \qquad(0\le\vartheta\le2m)
\]
shows that this does not improve the preceding estimate. In
particular, the global endpoint alone does not settle the stronger
localized conjecture below.

\subsection{The conjecture and its lower bound}

Define
\begin{equation}\label{eq:prime-field-localized-profile}
 \Phi_m(\vartheta)
 :=\max\left\{0,\min\left\{\frac{\vartheta-m}{2},\frac12\right\}\right\}
 =\begin{cases}
 0,&0\le\vartheta\le m,\\[1mm]
 (\vartheta-m)/2,&m\le\vartheta\le m+1,\\[1mm]
 1/2,&m+1\le\vartheta\le2m.
 \end{cases}
\end{equation}

We formulate the following conjecture.

\begin{conjecture}
\label{conj_main}
Fix $m\ge1$. For every $\varepsilon>0$, there is
$p_0(m,\varepsilon)$ such that, for every odd prime
$p\ge p_0(m,\varepsilon)$ and every nondegenerate quadratic form
$Q$ on $\F_p^{2m}$,
\[
 \sup_{j\in\F_p^\times}\Lambda_{S_j}(\vartheta)
 \le\Phi_m(\vartheta)+\varepsilon,
 \qquad 0\le\vartheta\le2m.
\]
\end{conjecture}

The proved and conjectured bounds agree on
$[0,m-\tfrac12]\cup[m+1,2m]$. At $\vartheta=m$, the proved
exponent is $1/4$, while the conjecture predicts zero. The following
example shows that the conjectured bound cannot be improved for
either Witt type.

\Needspace{12\baselineskip}
\begin{proposition} \label{prop:lower-bounds-Lambda}
Let $\epsilon_Q=1$ for split forms and $\epsilon_Q=-1$ for nonsplit
forms. For every $j\ne0$ and $0\le\vartheta\le2m$,
\begin{equation}\label{eq:affine-plane-localized-lower}
 \Lambda_{S_j}(\vartheta)
 \ge\max\left\{0,\frac12\log_p
 \frac{\min\{\lfloor p^\vartheta\rfloor,p^{m+1}\}}
      {p^m-\epsilon_Q}\right\}.
\end{equation}
In particular, uniformly over both Witt types,
\[
 \Lambda_{S_j}(\vartheta)
 \ge\Phi_m(\vartheta)-\frac{\log3}{2\log p}.
\]
For split forms, the sharper bound
$\Lambda_{S_j}(\vartheta)\ge\Phi_m(\vartheta)$ holds.
\end{proposition}

\begin{proof}
Choose $x_0\in S_j$. Since $Q(x_0)\ne0$, the orthogonal complement
$x_0^\perp$ is a nondegenerate space of dimension $2m-1$. Its Witt
index is $m-1$, so it contains a totally isotropic subspace $U$ of
that dimension. Thus, $x_0+U\subseteq S_j$. This includes $m=1$,
when $U=\{0\}$. Put $H=U^{\perp}$; then $|H|=p^{m+1}$.
For any nonempty $A\subseteq H$, define
\[
 g_A(x)=\1_A(x)\chi(B(x,x_0)).
\]
For every $u\in U$, one has $\widehat{g_A}(x_0+u)=|A|$. Therefore,
\[
 \frac{\|\widehat{g_A}\|_{L^2(S_j,d\sigma_j)}^2}{\|g_A\|_2^2}
 \ge\frac{p^{m-1}|A|}{|S_j|}
 =\frac{|A|}{p^m-\epsilon_Q}.
\]
Take $|A|=\min\{\lfloor p^\vartheta\rfloor,p^{m+1}\}$ to prove
\eqref{eq:affine-plane-localized-lower}. Since
$\lfloor p^\vartheta\rfloor\ge p^\vartheta/2$ and
$p^m+1\le(4/3)p^m$, the fraction is at least
$p^{\min\{\vartheta,m+1\}-m}/3$. This gives the uniform error term.
For split forms and $m\le\vartheta\le m+1$, use
\[
 \lfloor p^\vartheta\rfloor
 \ge p^\vartheta-1
 \ge p^\vartheta(1-p^{-m}).
\]
For $\vartheta\ge m+1$, take $A=H$, and for $\vartheta\le m$ use
$\Lambda_{S_j}(\vartheta)\ge0$.
\end{proof}

Consequently, under Conjecture~\ref{conj_main}, the localized
exponent converges to $\Phi_m(\vartheta)$ as $p\to\infty$, uniformly
in $\vartheta$, the nonzero level, and the choice of nondegenerate
form. Both split and nonsplit forms attain this limit.

\paragraph{Relation to weighted restriction.}
The localized norm is a weighted $L^2$ extension norm with weight
$\1_F$. This connects its concentration question with weighted
restriction problems. In the Euclidean Mizohata--Takeuchi problem,
the proposed control involves the largest mass of the weight along
a line, or inside a tube in a localized formulation (see
\cite[Conjecture~1.2]{CairoZhangMT}). Here the proposed bound depends
only on the cardinality of $F$ and takes the supremum over all sets
of that size. For related comparisons of localized spherical and
paraboloid extension estimates, see
\cite{IosevichZhangWeighted}.

\subsection{Reduction to the critical support size} \label{subsec:critical-support-scale}

The prime-field conjecture is equivalent to its assertion at
$\vartheta=m$. The following proposition proves this
reduction and keeps track of the normalization.

\begin{proposition}\label{prop:endpoint-propagation}
Suppose that $\sup_{j\ne0}\Lambda_{S_j}(m)\le\lambda$, where
$\lambda\ge0$, and put
\[
 c_Q:=\frac12\log_p\frac{p^{2m}}{|S_j|}
 =\frac12-\frac12\log_p(1-\epsilon_Qp^{-m}).
\]
This value is independent of the nonzero level $j$. Then
\[
 \sup_{j\ne0}\Lambda_{S_j}(\vartheta)
 \le
 \begin{cases}
 \min\{\lambda,\vartheta/2,c_Q\},&0\le\vartheta\le m,\\[1mm]
 \min\{\lambda+(\vartheta-m)/2,c_Q\},&m\le\vartheta\le2m.
 \end{cases}
\]
\end{proposition}

\begin{proof}
When $|F|\le p^m$, use the hypothesis and the trivial estimate
$\|\widehat g\|_{L^2(S_j,d\sigma_j)}\le\|g\|_1
\le |F|^{1/2}\|g\|_2$ for functions supported on $F$.
When $|F|>p^m$, choose $A\subseteq F$ uniformly among subsets of
size $p^m$, and write $g_A=g\1_A$. Then
\[
 \mathbb E g_A=\frac{p^m}{|F|}g,\qquad
 \mathbb E\|g_A\|_2^2=\frac{p^m}{|F|}\|g\|_2^2.
\]
Convexity of the squared Hilbert norm gives
\[
 \left(\frac{p^m}{|F|}\right)^2
 \|\widehat g\|_{L^2(S_j,d\sigma_j)}^2
 \le\mathbb E\|\widehat{g_A}\|_{L^2(S_j,d\sigma_j)}^2
 \le p^{2\lambda}\frac{p^m}{|F|}\|g\|_2^2.
\]
This proves the claimed bound for $|F|>p^m$. Plancherel supplies the
bound $c_Q$ without a support restriction.
\end{proof}

Since $c_Q=1/2+O(p^{-m}/\log p)$ uniformly over both Witt types,
Conjecture~\ref{conj_main} is equivalent to
\[
 \sup_Q\sup_{j\ne0}\Lambda_{S_j}(m)\longrightarrow0
 \qquad\text{as }p\to\infty\text{ through odd primes}.
\]
Indeed, the conjecture implies this limit because $\Phi_m(m)=0$.
Conversely, apply Proposition~\ref{prop:endpoint-propagation} with
$\lambda=\varepsilon/2$ and then take $p$ large enough that
$c_Q\le1/2+\varepsilon/2$. This proves the conjectured bound
uniformly throughout $0\le\vartheta\le2m$.

\subsection{From local restriction to global extension} \label{sec:restriction-extension-transfer}

For completeness, we record how a localized bound implies
global extension estimates away from the endpoint over prime fields. The
argument uses a dyadic decomposition and $\varepsilon$-removal;
it does not supply the endpoint proved earlier.

\begin{proposition}
\label{thm:restriction}
Fix $m\ge1$ and a function $\Lambda:[0,2m]\to[0,\infty)$.
Suppose that, for every $\delta>0$, there is $p_1(\delta)$ such
that, for every odd prime $p\ge p_1(\delta)$ and every
nondegenerate $Q$ on $\F_p^{2m}$,
\[
 \sup_{j\ne0}\Lambda_{S_j}(\vartheta)
 \le\Lambda(\vartheta)+\delta,
 \qquad 0\le\vartheta\le2m.
\]
Assume that
\[
 \varphi:=\sup_{0<\vartheta\le2m}
              \frac{\Lambda(\vartheta)}{\vartheta}<\frac12.
\]
Then, for every $r>2/(1-2\varphi)$, there is a constant $C$ such
that
\[
 R_{S_j}^*(2\to r)\le C
\]
for every odd prime $p$, every nondegenerate $Q$, and every $j\ne0$.
The constant may depend on $m$, $r$, $\varphi$, and the thresholds
$p_1(\delta)$, but not on $p$, $Q$, or $j$.
\end{proposition}

\begin{proof}
Fix $\delta>0$ and an odd prime $p\ge p_1(\delta)$. Let $h$ be
supported on a nonempty set $F\subseteq V$. For $|F|\ge2$, take
$\vartheta=\log_p|F|$ in the hypothesis. The definition of
$\Lambda_{S_j}$ gives
\begin{equation}\label{eq:sparse-restriction-from-Lambda}
 \|\widehat h\|_{L^2(S_j,d\sigma_j)}
 \le p^\delta|F|^\varphi\|h\|_{L^2(V)}.
\end{equation}
For $|F|=1$, the same inequality follows because
$|\widehat h(\xi)|=\|h\|_2$ for every $\xi$.

Let $1<t<2/(1+2\varphi)$. We claim that
\begin{equation}\label{eq:prime-to-two-restriction}
 \|\widehat g\|_{L^2(S_j,d\sigma_j)}
 \le C_{m,t,\varphi,\delta,p_1}\,p^\delta\|g\|_{L^t(V)}
\end{equation}
for every $g$, uniformly over odd primes, nondegenerate forms, and
nonzero levels. First suppose that $p\ge p_1(\delta)$. By
homogeneity, normalize $\|g\|_t=1$, so that $|g(x)|\le1$. Write
\[
 g=\sum_{k\ge0}g_k,\qquad
 g_k=g\1_{F_k},\qquad
 F_k=\{x:2^{-k-1}<|g(x)|\le2^{-k}\}.
\]
Since $|F_k|\le2^{(k+1)t}$, \eqref{eq:sparse-restriction-from-Lambda}
implies, for every nonzero $g_k$,
\[
 \|\widehat{g_k}\|_{L^2(S_j,d\sigma_j)} \le p^\delta2^{-k}|F_k|^{\varphi+1/2} \le p^\delta2^{t(\varphi+1/2)}
       2^{-k(1-t(\varphi+1/2))}.
\]
The exponent $1-t(\varphi+1/2)$ is positive. The triangle
inequality and the geometric series prove
\eqref{eq:prime-to-two-restriction} for these primes. For the
remaining primes, use
\[
 \|\widehat g\|_{L^2(S_j,d\sigma_j)}
 \le\|g\|_1
 \le p_1(\delta)^{2m(1-1/t)}\|g\|_t.
\]
This proves the claim for all odd primes.

Now fix $2/(1-2\varphi)<r<\infty$, and choose
$r'<t<2/(1+2\varphi)$, where $r'$ is conjugate to $r$.
By duality, \eqref{eq:prime-to-two-restriction} gives
\[
 R_{S_j}^*(2\to t')\lesssim_{m,t,\varphi,\delta,p_1}p^\delta
\]
for every $\delta>0$. The sphere kernel estimate in
Lemma~\ref{lem:kernel-linfty} gives
\[
 |(d\sigma_j)^\vee(x)|\lesssim p^{-m+1/2}\qquad(x\ne0),
\]
uniformly in $Q$ and $j$. Since $r'<t$ implies $r>t'$, apply
Lemma~\ref{lem:epsilon-removal} to odd prime fields to obtain the
claimed bound at $r$. The case $r=\infty$ follows directly from
Cauchy--Schwarz.
\end{proof}

For the conjectured function $\Phi_m$,
\[
 \sup_{0<\vartheta\le2m}\frac{\Phi_m(\vartheta)}{\vartheta}
 =\frac1{2(m+1)}.
\]
Indeed, the ratio increases from zero to this value on $[m,m+1]$
and decreases thereafter. Conjecture~\ref{conj_main} would therefore
imply $R_{S_j}^*(2\to r)\lesssim_{m,r}1$ for every
$r>2(m+1)/m$ over prime fields. Independently of that conjecture,
Theorem~\ref{thm:main-endpoint-three} proves the strong endpoint
and the full range $r\ge2(m+1)/m$ over all odd finite fields.

\subsection{Conditional pinned-distance estimates}

The distance results in Section~\ref{subsec:optimized-pinned-distances} are unconditional.
The next consequence uses Conjecture~\ref{conj_main} and is stated
only over prime fields.

\begin{theorem}
\label{thm:prime-local-pinned}
Assume Conjecture~\ref{conj_main}. Fix $m\ge1$ and
$\varepsilon>0$. For an odd prime $p$, let $Q$ be nondegenerate on
$\F_p^{2m}$ and let $E,F\subseteq\F_p^{2m}$ be nonempty.
Suppose that, for some $0\le\vartheta\le2m$,
\[
 |F|\le p^\vartheta,\qquad
 |E||F|\ge p^{2m+2\Phi_m(\vartheta)+\varepsilon}.
\]
Then, as $p\to\infty$, there is $F'\subseteq F$ with
$|F'|=(1-o(1))|F|$ such that
\[
 |\Delta_y(E)\setminus\{0\}|=(1-o(1))p
 \qquad\text{uniformly for }y\in F'.
\]
The error terms are uniform over $Q$, $E$, $F$, and $\vartheta$.
\end{theorem}

\begin{proof}
Apply Conjecture~\ref{conj_main} with error $\varepsilon/4$.
For all sufficiently large $p$, the definition of the localized
norm gives, for every $j\ne0$,
\[
 \|(\widehat{\1_E}\,d\sigma_j)^\vee\|_{L^2(F)}^2
 \le p^{2\Phi_m(\vartheta)+\varepsilon/2}
       \|\widehat{\1_E}\|_{L^2(S_j,d\sigma_j)}^2.
\]
Recall from the preceding section that
$\Psi_j(y)=\sum_{\xi\in S_j}\widehat{\1_E}(\xi)\chi(B(y,\xi))$.
Since $\Psi_j=|S_j|(\widehat{\1_E}\,d\sigma_j)^\vee$ and
$|S_j|$ is independent of $j\ne0$, summing the last bound and
applying Plancherel yield
\[
 \begin{aligned}
 &p^{-2m}\sum_{j\ne0}\sum_{y\in F}|\Psi_j(y)|^2
 \le p^{-2m}|S_j|^2p^{2\Phi_m(\vartheta)+\varepsilon/2}
       \sum_{j\ne0}\|\widehat{\1_E}\|_{L^2(S_j,d\sigma_j)}^2\\
 &\qquad\qquad  \le |S_j|p^{2\Phi_m(\vartheta)+\varepsilon/2}|E| \lesssim p^{2m-1+2\Phi_m(\vartheta)+\varepsilon/2}|E|.
 \end{aligned}
\]
Insert this into \eqref{eq:optimized-pinned-energy-start}. The
assumed size condition gives
\[
 \frac{\mathcal I^\times(E;F)}{|E|^2|F|/p}
 \le1+C_m\frac{p^{2m+2\Phi_m(\vartheta)+\varepsilon/2}}{|E||F|}
       +\frac3p
 \le1+C_mp^{-\varepsilon/2}+\frac3p.
\]
Moreover, $\Phi_m\ge0$, so $|E||F|\ge p^{2m+\varepsilon}$.
Lemma~\ref{lem:zero-distance-incidence} implies
\[
 \frac{\nu_{E,F}(0)}{|E||F|}
 \le p^{-1}+2p^{-\varepsilon/2}=o(1).
\]
Lemma~\ref{lem:almost-every-pin-energy-criterion} now gives the
conclusion.
\end{proof}

\begin{corollary}
\label{cor:prime-distance-conjecture}
Assume Conjecture~\ref{conj_main}. For every $m\ge1$ and
$\varepsilon>0$, every set $E\subseteq\F_p^{2m}$ with
$|E|\ge p^{m+\varepsilon}$ has a subset $E'\subseteq E$ such that
\[
 |E'|=(1-o(1))|E|,\qquad
 |\Delta_y(E)\setminus\{0\}|=(1-o(1))p
 \quad(y\in E'),
\]
as $p\to\infty$ through odd primes, uniformly over nondegenerate
$Q$ of either Witt type.
\end{corollary}

\begin{proof}
Write $|E|=p^{m+t}$, where $t\ge\varepsilon$. Since
$2\Phi_m(m+t)=\min\{t,1\}\le t$, one has
\[
 |E|^2=p^{2m+2t}
 \ge p^{2m+2\Phi_m(m+t)+\varepsilon}.
\]
Apply Theorem~\ref{thm:prime-local-pinned} with $F=E$ and
$\vartheta=m+t$.
\end{proof}

Thus, the conjectured local estimate would lower the prime-field
threshold for almost every pin from $m+1/3$ to $m+\varepsilon$,
including the threshold $1+\varepsilon$ in dimension two. This
conclusion remains conditional; the global endpoint and the
direct distance estimates proved above do not supply the missing
local bound at support size $p^m$.
\section*{Acknowledgements}
We thank Doowon Koh for helpful discussions and for sharing his ideas.
The second author also thanks Danqing He and Shanlin Huang for useful discussions.

\end{document}